\documentclass[11pt]{preprint} 

\usepackage[a4paper,innermargin=1.2in,outermargin=1.2in,
bottom=1.5in,marginparwidth=1in,marginparsep
=3mm]{geometry}
\usepackage{bigints}
\usepackage[T1]{fontenc} 
\usepackage{fourier} 
\usepackage[english]{babel} 
\usepackage{amsmath,amsfonts,amsthm} 
\usepackage{tikz-cd}

\usepackage{shuffle}
\usepackage{ amssymb }
\usepackage{fancyhdr} 
\usepackage{comment}
\usepackage{stmaryrd}

\usepackage{cprotect}
\usepackage{hyperref}
\usepackage{mathrsfs}
\usepackage{mhequ}
\usepackage{microtype}
\usepackage{esint}

\usepackage{xstring}

\usepackage{orcidlink}

\colorlet{darkblue}{blue!90!black}
\colorlet{darkred}{red!90!black}

\newcommand{\opP}{\operatorname{P}}

\newcommand{\mbG}{\mathbb{G}}

\newcommand{\mbN}{\mathbb{N}}

\newcommand{\mbR}{\mathbb{R}}

\newcommand{\RR}{\mathbb{R}}

\newcommand{\NN}{\mathbb{N}}

\newcommand{\frakg}{\mathfrak{g}}

\newcommand{\dd}{\mathop{}\!\mathrm{d}}

\newcommand{\mcD}{\mathcal{D}}

\newcommand{\mcP}{\mathcal{P}}

\newcommand{\mfB}{\mathfrak{B}}

\newcommand{\mfg}{\mathfrak{g}}

\newcommand{\mfr}{\mathfrak{r}}
\newcommand{\mfs}{\mathfrak{s}}

\renewcommand{\bf}{\mathbf{f}}

\newcommand{\fraks}{\mathfrak{s}}

\usepackage{float}

\newcommand{\eps}{\varepsilon}

\newcommand{\E}{\mathbb{E}}

\newcommand{\dil}{\operatorname{D}}

\newcommand{\vertiii}[1]{{\left\vert\kern-0.25ex\left\vert\kern-0.25ex\left\vert #1 
		\right\vert\kern-0.25ex\right\vert\kern-0.25ex\right\vert}}

\newtheorem{theorem}{Theorem}[section]

\newtheorem{corollary}[theorem]{Corollary}
\newtheorem{claim}[theorem]{Claim}
\newtheorem{lemma}[theorem]{Lemma}
\newtheorem{definition}{Definition}
\newtheorem{prop}[theorem]{Proposition}
\newtheorem{assumption}[theorem]{Assumption}
\theoremstyle{remark}
\newtheorem{remark}[theorem]{Remark}

\newcommand{\id}{\mathrm{id}}
\numberwithin{equation}{section} 
\numberwithin{figure}{section} 
\numberwithin{table}{section} 

\colorlet{symbols}{blue!90!black}
\colorlet{testcolor}{green!60!black}

\makeatletter

\DeclareRobustCommand{\TitleEquation}[2]{\texorpdfstring{\StrLeft{\f@series}{1}[\@firstchar]$\if%
		b\@firstchar\boldsymbol{#1}\else#1\fi$}{#2}}

\makeatother

\usepackage{enumitem}
\setlist[enumerate]{itemsep=0pt, topsep=0.2pt}
\setlist[itemize]{itemsep=0pt, topsep=0pt}

\colorlet{darkblue}{blue!90!black}
\colorlet{darkgreen}{green!50!black}
\colorlet{darkred}{red!90!black}

\newcommand{\horrule}[1]{\rule{\linewidth}{#1}} 

\title{	
\horrule{0.5pt} \\[0.4cm] 
\large Weighted Besov Spaces on Homogeneous Lie Groups and Applications to Parabolic Anderson Models 
 \\ 
\horrule{0.5pt} \\[0.5cm] 
}
\author{Harprit Singh \orcidlink{0000-0002-9991-8393} }
\institute{
ETHZ, Z\"urich, Switzerland, \email{harprit.singh@math.ethz.ch}
 }
 
\allowdisplaybreaks[2]
\begin{document}

\maketitle 
%

\begin{abstract}
We develop an intrinsic theory of weighted, inhomogeneous Besov spaces on general homogeneous Lie groups without recourse to group-specific arguments. Starting from a definition in terms of localised test functions, we establish an equivalent, wavelet-like, multiscale characterisation.
 This provides a unified mechanism for deriving Besov embeddings, a Taylor-remainder characterisation, Young-type product estimates, Schauder estimates for convolution semigroups, and a weighted Kolmogorov criterion for random distributions.

We apply this framework to parabolic Anderson-type equations associated with positive Rockland operators and with singular initial data. For space-time noise, our results cover the Young regime. For purely spatial noise, we also establish well-posedness in the first singular regime using a variant of the Cole-Hopf transform for Rockland operators. Both applications rely on a mild sewing lemma that accommodates time-dependent Banach spaces and increments with a singularity at the initial time.
\end{abstract}

\tableofcontents

\section{Introduction}
Besov spaces play a central role in the study of partial and stochastic partial differential equations. When the equation involves differential operators that are hypo-elliptic, such as sub-Laplacians, the kinetic operator or Rockland operators, it is often fruitful to interpret the underlying space as a homogeneous Lie group. 
Besov-type spaces on such groups have been studied from several complementary perspectives; see, for example, \cite{folland_stein_82_hardy, Besov_def, FurioliMelziVeneruso, Führ, LPtheorem, hu2025besovtriebellizorkinspaceshomogeneous}. 
Nevertheless, many estimates required for the well-posedness theory of PDEs, and especially pathwise SPDEs, continue to be established on a case-by-case basis, see  \cite{Tindel_ito, Tindel_Big} for the Heisenberg group and \cite{KineticZhu} for the kinetic group. These constructions, as well as classical approaches based on the representation theory of the underlying group, cf.\ \cite{fischer_ruzhansky_18_quantisation}, become technically involved as the structure of the group becomes complicated, requiring bespoke development for each equation.

This paper provides an intrinsic and comparatively direct framework for establishing properties of weighted, inhomogeneous Besov spaces $B^{\alpha,w}_{p,q}(\mathbb{G})$ on a general homogeneous Lie group $\mathbb{G}$, needed in the analysis of pathwise SPDEs. 
We demonstrate the scope of this framework by giving self-contained well-posedness proofs for infinite volume parabolic Anderson models 
\begin{equ}\label{eq:introPam}
(\partial_t+ \mathcal{R})u= u\cdot \xi, \qquad u(0)=u_{0}
\end{equ}
associated with a positive Rockland operator $\mathcal{R}$ 
on an arbitrary homogeneous Lie group and allowing singular initial conditions $u_{0}$, including Dirac masses.

Our first result on \eqref{eq:introPam}, Theorem~\ref{thm:PAM_YOung}, establishes pathwise well-posedness in the Young regime, i.e. in the regularity regime where the product on the right hand side can be classically defined, and thereby substantially extends the Heisenberg-group theory of \cite{Tindel_Big}, see Remark~\ref{rem:strengthening} for a precise comparison. Our second application, Theorem~\ref{thm:first singular},  shows well-posedness of \eqref{eq:introPam} beyond the Young regime in the case of purely spatial noise, thereby extending \cite[Thm.~6.4]{MS25} in several directions, see Remark~\ref{rem:singular_comparision}.

\paragraph{The functional analytic framework}

The first part of the paper develops an intrinsic theory of weighted, inhomogeneous Besov spaces $B^{\alpha,w}_{p,q}(\mathbb G)$ on a general homogeneous Lie group $\mathbb G$, starting from a definition in terms of localised test functions, Definition~\ref{def:Besov}. The central functional-analytic result is Theorem~\ref{prop:equiv}, which gives an equivalent multiscale characterisation of $B^{\alpha,w}_{p,q}(\mathbb G)$ based on a single pair of test functions $(\phi,\rho)$ satisfying Assumption~\ref{ass:pair}.

Pairs of this type have previously been used in proofs of the reconstruction
theorem in regularity structures
\cite{Hai14,friz_hairer_20_introduction,MS25}. Here their role is to provide a common mechanism for developing the weighted Besov space theory, which is agnostic to the details of the group structure of $\mathbb{G}$ and in particular completely avoids its representation theory.
The characterisation of Theorem~\ref{prop:equiv} is conceptually close to the wavelet approach to Besov spaces on $\mathbb{R}^d$, cf.~\cite{meyer1992wavelets}.
In contrast to existing works on wavelet-type characterisations on homogeneous Lie groups, see \cite{Führ,hu2025besovtriebellizorkinspaceshomogeneous}, the functions $(\varphi,\rho)$ are compactly supported, which is important for working with the general class of weights of Definition~\ref{def:weight}, and their construction is elementary, see Remark~\ref{rem:existence}. 
Among the main functional-analytic results derived from
Theorem~\ref{prop:equiv} are the following:
\begin{enumerate}
\item For $\alpha>0$ we recover an equivalent Taylor remainder characterisation of $B^{\alpha,w}_{p,q}$ in Theorem~\ref{prop:difference}. 
\item A Young-type multiplication theorem, Theorem~\ref{prop:young}, showing that pointwise multiplication of functions extends to Besov spaces of regularity $\alpha,\beta\in \mathbb{R}$ as long as $\alpha+ \beta>0$.
This is related, but distinct, to the para-product estimates for positive regularity Besov spaces discussed on the Heisenberg group in \cite{Tindel_ref_gallager} and more generally in \cite{gallager_sire}, see also \cite{Gallager_phase}.
\item Convolution estimates for singular kernels, see Proposition~\ref{prop:schauder}. These imply Schauder estimates for heat semi-groups associated to positive Rockland operators,
as it is well known that their heat kernels satisfy the assumptions we make on kernels.
\end{enumerate}

\noindent We also derive Besov embeddings with respect to the parameters $\alpha,p,q$ and the weight $w$, see Corollary~\ref{prop:embedding}, 
and \cite[Thm.~5.3]{LPtheorem} for analogous embeddings for
unweighted Besov spaces, starting from a spectral characterisation.
Lemma~\ref{lem:kolmogorov} provides a Kolmogorov criterion, which lets one verify when a random field has a modification that belongs to a Besov space $B^{\alpha,w}_{p,q}(\mathbb{G})$.

\begin{remark}
Throughout, we work with a broad class of admissible weights, see
Definition~\ref{def:weight}. We specify explicitly where additional
properties of the weight are required. Formulating results in this generality is particularly useful in the pathwise study of SPDEs in unbounded spaces such as $\mathbb{G}$, where stationary noises become unbounded.
\end{remark}

\begin{remark}
Closely related functional-analytic results were recently obtained in
\cite{Tindel_Big} for the Heisenberg group equipped with its sub-Laplacian.
The approach in that work is based on the coordinate-wise framework of
\cite{Tindel_ref1} and builds on the paraproduct estimates of
\cite{Tindel_ref_gallager}, and thereby heavily uses the group-specific structure.
It therefore seems challenging to pursue said approach in the generality of this article.
\end{remark}

Let us emphasise that a principal advantage of the intrinsic and
group-uniform viewpoint developed here is that many arguments can be organised
in close analogy with their Euclidean counterparts, without introducing
separate coordinate- or representation-theoretic constructions for each
group.\footnote{
In the case of $\mathbb{G}=\mathbb{R}^d$, while the exact arguments and definitions used here, in particular the characterisation of Theorem~\ref{prop:equiv} for $p\neq \infty$, do not seem to be present in the existing literature, the results on Besov spaces are known in that case. 
} The main contribution therefore lies not in each
individual argument, but in the unified approach that yields concise and
largely self-contained proofs based on the standard structural properties of
homogeneous Lie groups recalled in
Section~\ref{sec:analysis_background}.

\paragraph{PAM in the Young Regime}
We apply the functional analytic results about weighted Besov spaces in Theorem~\ref{prop:fixedpoint} to 
establish well-posedness of PAM-type equations 
\begin{equ}\label{eq:illustrate}
(\partial_t+ \mathcal{R})u= \sum_{d(I)<m} c_I X^{I}u\cdot\xi^{(I)}, \qquad u(0)=u_0\ ,
\end{equ}
for $\mathcal{R}$ a positive Rockland operator, $c_I\in \mathbb{R}$ and $\xi^{(I)}$ space-time distributions of space-time regularity such that the equation belongs to the Young regime, 
see Section~\ref{sec_applications to SPDE} for the precise conditions, and with singular initial condition.
%
%

%

Combining
Theorem~\ref{prop:fixedpoint} with the Kolmogorov criterion, Lemma~\ref{lem:kolmogorov} yields the following precise special case.
We use the scaling notation 
$\phi^{\lambda}(z):=\frac{1}{\lambda^{|\mfs|}}\phi \left(\frac{1}{\lambda}\cdot  (z)\right)$ where $\cdot$ is the intrinsic dilation on $\mathbb{G}$ and $|\fraks|$ the homogeneous dimension, see Section~\ref{sec:analysis_background}.

\begin{theorem}\label{thm:PAM_YOung}

Let $\alpha\leq 0< H\leq 1$ and assume $W: [0,T]\to \mathcal{D}'(\mathbb{G})$ is a stationary under left-translations random process satisfying for any smooth compactly supported function $\phi$
and $q\in \mathbb{N}$
$$ 
\E\big[\big|W_t(\phi^{\lambda})-W_s(\phi^\lambda)\big|^q\big]\lesssim C_{\phi,q}|t-s|^{Hq} \lambda^{ \alpha q}
$$
uniformly over $0\leq s<t\leq T$ and $\lambda\in (0,1]$.
If $\alpha/m +H>1/2 $, then the equation 
\begin{equ}\label{Pam_Young}
(\partial_t+ \mathcal{R})u= u\dot{W}, \qquad u(0)=u_{0}
\end{equ}
is globally well-posed \footnote{We refer to Section~\ref{sec:fixed point} for the precise function spaces in which this and the more general \eqref{eq:illustrate} is well-posed.} for any initial condition $u_0\in B^{-m\nu, w}_{p, \infty}$ 
for $p\in [1,\infty]$,
 $ \nu\in [0,H+\alpha/m) $
 and any weight $w$ such that
$|w(x)|\leq C \exp(C|x|)$ for some $C>0$ .
\end{theorem}
\begin{remark}\label{YoungRegime_heuristic}
We observe that the condition $\alpha/m +H>1/2 $ corresponds to the usual heuristic for the right-hand side of the equation to be well defined. If one uses intrinsic function spaces on $\mathbb{G}\times \mathbb{R}$ and that $\partial_t+\mathcal{R}$ is $m$ regularising on this space, it is simply a rewriting of the condition that if we denote by $\bar{\alpha}= \alpha+m (H-1)$ the (intrinsic) space-time regularity\footnote{Of course there is a slight loss of regularity due to the use of Kolmogorov's criterion, however since all inequalities are strict this can be absorbed.} of $\dot{W}$, that 
$$  \bar{\alpha} + (\bar{\alpha}+m)>0\ .$$
\end{remark}

\begin{remark}\label{rem:strengthening}
Compared with \cite{Tindel_Big}, Theorem~\ref{thm:PAM_YOung} replaces the structurally simple
Heisenberg group and its sub-Laplacian by an arbitrary homogeneous Lie group
and a general positive Rockland operator. Furthermore, instead of  regular initial conditions, it allows for singular initial conditions which include the Dirac delta measure. 
Lastly, the class of noises considered here contains those considered therein.

%
%
 
%
%
\end{remark}

\paragraph*{PAM in the first singular regime}
We also obtain well-posedness in the first singular regime for noises which only depend on space. 
In order to specify this regime, let $r_I\in\mathbb{R}$ be the coefficients of the Rockland operator 
\begin{equ}\label{eq:form rockland}
 \mathcal{R}=\sum_{d(I)=m}r_I X^I\  
\end{equ}
 when written in terms of left-invariant operators $X^I$, see Section~\ref{sec:analysis_background}, 
and set
\begin{equation}
 \delta_{\mathcal{R}}:=\min\{s_j:\text{there is an }I\text{ with }r_I\neq0,
                    \ |I|\geq 2,\text{ and }i_j>0\}.
 \label{eq:delta-R}
\end{equation}

\begin{theorem}\label{thm:first singular}
For $\zeta\leq 0$, let $\xi$ be a centred stationary Gaussian random field on $\mathbb{G}$ with covariance\footnote{This should of course be understood 
rigorously as $\E[\xi(\psi)\xi(\phi)]=\langle\psi, \phi*C\rangle\ $}
\begin{equ}\label{eq:covariance_noise}
 \E[\xi(x)\xi(y)]=C(y^{-1}x)
 \end{equ}  where $C\in \mathcal{D}'(\mathbb{G})$ is a distribution satisfying $|C(\phi^{\lambda})|\lesssim \lambda^{2\zeta}$ uniformly in $\lambda\in (0,1]$ for every $\phi\in C_c^\infty(B_1)$ and which
 away from the origin agrees with a smooth function satisfying $|X^{I} C(x)|\lesssim_I 1+|x|^{2\zeta-d(I)}$ uniformly over $x\in\mathbb{G}\setminus\{e \}$.


Set $\zeta_\star
:=
\max\left\{
-\frac{2m}{3},
-\frac{m+\delta_{\mathcal R}}{2},
-\frac{2m+|\fraks|}{4},
-\frac{|\fraks|}{2}-\delta_{\mathcal R}
\right\}
$
and assume that 
 $\zeta\in\big( \zeta_\star, -\frac{m}{2}\big]$ and let $\xi_\eps=\xi* \rho^{\eps}$ for $\rho\in C_c(B_1)$ such that $\int\rho=1$. There exist constants $C_{\eps}$ such that for any $\eps>0$ the equation 
\begin{equ}\label{eq:singular}
(\partial_t+ \mathcal{R})u_\eps= u_\eps(\xi_\eps-C_\eps) \ .
\end{equ}
is well-posed\footnote{We refer to Section~\ref{sec:proof_of_singularPAM} for the precise function spaces.}  for
 initial conditions $u_0\in B_{p,\infty}^{-m\nu,w}$ 
 for any $p\in [1,\infty]$,
 $\nu\in [0,1+\zeta/m)$ and any weight $w$ such that
$|w(x)|\leq C \exp(C|x|)$ for some $C>0$ and such that furthermore,
 there exists $u:(0,T]\to C(\mathbb{G})$ independent of the choice of $\rho$ such that $u_\eps \to u$ locally uniformly as $\eps\to 0$ in probability.

\end{theorem}

\begin{remark}\label{rem:singular_comparision}
This extends \cite[Thm.~6.4]{MS25}, which is formulated on a compact
quotient of Carnot groups, to the full group and to rough initial data. The value of $\zeta_*$ is determined by
the chosen proof strategy which proceeds via a variant of the Cole-Hopf transform for Rockland operators, avoiding the use of regularity structures.
\end{remark}


\begin{remark}
A conceptual point stressed in both \cite{Tindel_Big, Tindel_ito} is the identification of `critical exponents' on the Heisenberg group.
Let us observe that these conditions are exactly those one obtains by adapting the usual power counting of singular SPDEs \cite{Hai14} by replacing dimension by the homogeneous dimension, the degree of differential operators by the homogeneous degree and the regularity by the notion of regularity associated to intrinsic distribution spaces. 

The work \cite{Tindel_Big} identifies exponents
 for the equation  \eqref{eq:illustrate} to belong to the Young regime.
Theorem~\ref{thm:PAM_YOung} and Remark~\ref{YoungRegime_heuristic} in particular identify the Young regime on general groups and for general Rockland operators.  
The work \cite{Tindel_ito} considers the case of white in time noise and introduces the analogue of Dalang's condition identified for the Heisenberg group in \cite{Tindel_ito}. This is exactly the threshold for the noise to be subcritical, \cite{Hai14}, when specialised to white in time noise. We refer to \cite{MS25} for a discussion of how going from $\mathbb{R}^d$ to a general group $\mathbb{G}$ affects the theory of regularity structures.


%
\end{remark}

\paragraph{Arguments for the SPDE applications}
Both Theorem~\ref{thm:PAM_YOung} as well as Theorem~\ref{thm:first singular} follow from Theorem~\ref{prop:fixedpoint} which shows well-posedness of parabolic Anderson models of the more general \eqref{eq:illustrate} Besides using the functional analytic tools established in the main part of the text, this is achieved by a mild sewing lemma, Lemma~\ref{lem:sewing},
in the spirit of \cite{Gub_tindel_sewing}, which however allows for time-varying target spaces and singular increments. This accommodates increments with values in weighted spaces where the weight depends on time which is crucial for spatially unbounded noises. We work with singular increments to allow for singular initial conditions.


While Theorem~\ref{thm:PAM_YOung} is simply a special case of Theorem~\ref{prop:fixedpoint} combined with the Kolmogorov criterion, Lemma~\ref{lem:kolmogorov}, the proof of Theorem~\ref{thm:first singular} employs a variant of the Cole-Hopf transform.
Instead of considering solutions $u_\eps$ directly, we set 
$
v_\varepsilon$ to be the solution of $\mathcal Rv_\varepsilon=\xi_\varepsilon-v_\varepsilon$ on $\mathbb{G}$
and consider 
 $w_\eps=e^{-v_\eps}u_\eps$. Then it turns out, that $w_\eps$ satisfies an equation of the form \eqref{eq:illustrate} where the
 $\xi^{(I)}$ are explicit in terms of Bell-polynomials of derivatives of $v_\eps$. If $\mathcal{R}$ were only the Laplacian on $\mathbb{R}^d$, or the sub-Laplacian on a Carnot group, it would be straightforward to see for which regularity regimes of the noise the transformed equation belongs to the Young regime, cf.\ \cite{hairer_labbe_15_simple}. In our setting the transformed equation is more complicated and can contain many terms which in principle need renormalisation. We observe, nevertheless, that when $\zeta_\star<-m/2$ there is still a non-empty regularity regime where only quadratic terms require renormalization.

%
%

\begin{remark}
Theorem~\ref{thm:first singular} can be seen as part of ongoing developments to understand singular SPDEs, beyond constant coefficient parabolic equations, see \cite{Hai14, gubinelli_imkeller_perkowski_15}  as well as the subsequent systematic treatment 
\cite{BHZ19, BCCH20, CH16},
 which combined with \cite{BB16, Sin23, BSS25} extends to variable-coefficient parabolic setting.
It would be interesting to see how a systematic treatment of hypo-equation would look like. A first step in this direction was recently carried out in \cite{MS25}.

Other geometries in which singular SPDEs have been considered are Riemannian manifolds \cite{BB16, DDD19, mouzard_22_weyl, HS23manifolds}, 
 random/oscillatory coefficients \cite{CS25, clozeau}/\cite{HSper, CX23, CFX26, HSper2} and fractals \cite{fractals_1, fractals_2}.
We also mention \cite{BGHZ22, CCHS22, CCHS22b, CSYYM} where the target space is geometric.
\end{remark}

\paragraph{Acknowledgements}
HS would like to thank Avi Mayorcas for a discussion on this work and gratefully acknowledges financial support from the Swiss National Science
Foundation (SNSF), grant numbers
225606
 and 239497.
ChatGPT was used to proofread the manuscript which resulted in several minor corrections.

\subsection{Background on Homogeneous Lie Groups}\label{sec:analysis_background}
In this section we collect some standard background on homogeneous Lie groups, cf.\ the standard works \cite{folland_stein_82_hardy,fischer_ruzhansky_18_quantisation}, as well as \cite{MS25} for the same notation as here.
%
The following definition provides the setting of our article.
\begin{definition}
	A dilation on a Lie algebra $\frakg$ is a group of algebra automorphisms $\{\dil_r\}_{r>0}$ of the form \mbox{$\dil_r X= \exp (\log r \cdot \fraks) X $} with $\fraks : \frakg\to \frakg$ being a diagonalisable linear operator, which we assume without loss of generality has $1$ as its smallest eigenvalue.
	
	A homogeneous Lie group $\mbG$ is a simply connected, connected Lie group where its Lie algebra $\frakg$ is endowed with a family of dilations $\{\dil_r\}_{r>0}$.\footnote{Recall that if $\mfg$ admits a family of dilations, it is nilpotent.} For $r>0$, we define the group automorphism $$x\mapsto r\cdot x:= \exp \circ \dil_r\circ \exp^{-1} x \ .$$ 
	
	A homogeneous norm on $\mbG$ is a continuous function $|\cdot|: \mbG\to [0,\infty)$ satisfying $|x|=|x^{-1}| $ and $|r\cdot x|=r |x| $ for all $x\in \mbG$, $r\in \RR_+$
	as well as 
	$$|x|=0 \qquad \text{ if and only if } \qquad x=e\ .$$
\end{definition}

All homogeneous norms are known to be mutually equivalent, see \cite[Prop.~ 1.5 \& Prop.~ 1.6]{folland_stein_82_hardy}, and by \cite{HebischSikora90} we from now on fix such a norm which furthermore is smooth away from the origin and satisfies the  triangle inequality
$$
|xy|\leq |x|+|y| \qquad \text{ for any }x,y \in \mbG \ .
$$

A homogeneous norm naturally induces a topology generated by the (norm) open sets which agrees with the topology of $\mbG$ as a Lie group, see \cite[Sec.~3.1.6]{fischer_ruzhansky_18_quantisation}. From now on, we will always assume that $\mbG$ is equipped with this topology and the Borel $\sigma$-algebra. 

Furthermore, we denote by $\{X_{j}\}_{j=1}^d\in \mathfrak{g}$ a basis of eigenvectors of $\fraks$ with eigenvalues $1=\fraks_1\leq \fraks_2 \leq ... \leq \fraks_d$ and such that 
\begin{equation}\label{eq:eigenvectors}
	\fraks X_j = \fraks_j X_j \ .
\end{equation}
Given a measurable subset $E\subset \mbG$ we write $|E|$ for its Haar measure which we assume to be normalized such that the set $B_1:=\{x\in \mbG\ : \  | x| \leq 1\}$ has measure $1$. In integrals we use the standard notation $\dd x$. 
%
%

We define $|\fraks|:=\text{trace} (\fraks)$ as the homogeneous dimension of $\mbG$, since  for any measurable subset $E\subset \mbG$ and $r> 0$, one has
\begin{equation*}
	|r\cdot E|= r^{|\fraks|} |E| \ . 
	\end{equation*}
We also define balls $B_r(x):= \{y\in \mbG\ : \  | x^{-1}y| < r\}$ of radius $r> 0$ centered at $x\in \mathbb{G}$.
 Note that due to the non-commutativity of $\mbG$, in general $|x^{-1}y|\neq |yx^{-1}|$. 

\subsubsection{Derivatives and Polynomials}
We identify $\frakg$ with the left-invariant vector fields $\frakg_L$ on $\mbG$ and write $\frakg_R$ for the right-invariant ones. 
We write $X_i$ for the basis elements as in \eqref{eq:eigenvectors} seen as elements of $\frakg_L$ and write $Y_i$ for the basis of $\frakg_R$
satisfying $Y_i |_e= X_i |_e$ . Thus we can write 
$$X_jf(y)= \partial_t f(y\exp(tX_j))|_{t=0}\qquad  
Y_jf(x)= \partial_t f(\exp(tX_j) x)|_{t=0}
$$
for any smooth function $f\in C^\infty(\mbG)$.

	%
	%
	A map $P: \mbG \to \mathbb{R}$ is called a polynomial if $P\circ \exp : \mfg \rightarrow \mbR$ is a polynomial on $\mfg$.\footnote{Recall that the space of polynomial functions on $\frakg$ is canonically isomorphic to $\bigoplus_{n} (\frakg^\ast)^{\otimes_s n}$ where $\otimes_s$ denotes the symmetric tensor product.} Let $\zeta_i$ be the basis dual to the basis $X_i$ of $\frakg$. We set $\eta_j= \zeta_j\circ \exp^{-1}$, which maps $\mbG$ to $\mathbb{R}$.
	Note that $\eta= (\eta_1,...,\eta_d)$ forms a global coordinate system \footnote{Occasionally we use the corresponding notation $\frac{\partial}{\partial\eta_i}$.} and furthermore any polynomial map on $\mbG$ can be written in terms of coefficients $a_I \in \mbR$ as 
	$$P= \sum_I a_I \eta^I$$
	with the sum running over a finite subset of $\NN^d$ and where for a multi-index $I=(i_1,...,i_d)\in \NN^d$ we write
	$\eta^I= \eta_1^{i_1}\cdot...\cdot \eta_d^{i_d}$. 
	Define $d(I)=\sum_j \fraks_j i_j$ and $|I| = \sum_j i_j$, we call $\max \{ d(I)\  : \ a_I \neq 0\}$ the homogeneous degree and $\max\{ |I| \  : \ a_I \neq 0\}$ the isotropic degree of $P$. For $a\in \mathbb{R}$, we denote by $\mathcal{P}_a$ the space of polynomials of homogeneous degree strictly less than $a$ and define $\triangle= \{d(I)\in \mathbb{R} \ : \ I\in \mathbb{N}^d \}$. \footnote{We point out a possibly counter-intuitive quirk of our definition; for $k\in \triangle$  the set $\mcP_k$ does \textit{not} contain polynomials of degree $k$ but only those of degree less than $k$. 
	}
	%
%
%
	%
%
		Recall that $\mathcal{P}_a$ is invariant under right and left-translations (which is in general not true if one replaces homogeneous degree by isotropic degree).
	%
	%
	%
	For a multi-index $I=(i_1,...,i_d)\in \mbN^d$ we introduce the notation $X^I = X_1^{i_1}... X_d^{i_d}$. Note that the order of the composition matters, we use the reverse convention $Y^I = Y_1^{i_d}... Y_d^{i_1}$ for the right invariant vector fields.
	\begin{definition}\label{def:taylor}
		For a smooth function $f:\mbG\to \mathbb{R}$, a point $x\in \mbG$ and $a\in (0,\infty)$, we define the left Taylor polynomial of homogeneous degree (less than) $a$ of $f$ at $x$ to be the unique polynomial $\opP^a_x[f]\in \mathcal{P}_a$ such that $X^I \opP^a_x[f](e)= X^I f(x)$ for all $I$ such that $d(I)<a$. 
		
		We also define the centred Taylor polynomial $\tilde{\opP}^a_x[f](y):= \opP^a_x[f](x^{-1}y)$, which alternatively is characterised by requiring $X^I \tilde{\opP}^a_x[f](x)= X^I f(x)$ for $d(I)<a$. 
	\end{definition}
	 
	%
	To see that the above definition makes sense, recall that for any $x\in \mathbb{G}$ the maps from $\mathcal{P}_a\to \mathbb{R}^{\dim \mathcal{P}_a}$ 
		\begin{enumerate}
			\item $P \mapsto \left\{ \left(\frac{\partial}{ \partial_\eta}\right)^I P(x)\right\}_{d(I)< a}$ ,
			\item $P \mapsto \left\{ X^I P(x)\right\}_{d(I)< a}$ ,
			\item $P \mapsto \left\{ Y^I P(x)\right\}_{d(I)< a}$ ,
		\end{enumerate}
		are isomorphisms, cf. \cite[Prop.~1.30]{folland_stein_82_hardy}. It furthermore holds, see \cite[Prop.~1.29]{folland_stein_82_hardy}, that for every $i\in \{1,...,d\}$ there exist homogeneous polynomials $p_{i,j}$, $q_{i,j}$ of homogeneous degree $\fraks_j-\fraks_i$ such that 
		\begin{equ}\label{eq:re_expand_basis}
		X_i= \sum_{j} p_{i,j} Y_j, \qquad Y_i= \sum_{j} q_{i,j} X_j \ .
		\end{equ}
	
	%
	
	%
	%
	%
	
	%
	The following is Taylor's theorem from \cite[Thm.~2.13]{MS25}.
	\begin{theorem}\label{th:taylor}
		For each $a\geq 0$ and every $f\in C^\infty(\mbG)$ it holds that
		$$ f(xy)-\opP^a_x[f](y)= \sum_{|I|\leq [a]+1, d(I)\geq a} \int_{\mbG} X^{I}f(xz) Q^I(y, \dd z)  \ ,$$
		where for each multi-index $I$ and $y\in \mbG$ the measure $Q^I(y, \,\cdot\,)$ is supported on $B_{ \beta^{[a]+1} |y|}(e)$ for some $\beta>0$ depending only on $\mbG$ and satisfies $\int_{\mbG} |Q^I(y, \dd z)| \lesssim |y|^{d(I)}$.
	\end{theorem}
	\begin{remark}\label{rem:taytay} Note that the common version of Taylor's Theorem on homogeneous Lie groups, cf.\ \cite[Thm.~1.37]{folland_stein_82_hardy}
		\begin{equation}\label{eq:common_form_taylor}
			|f(xy)-\opP^a_x[f](y)|\lesssim_a \sum_{|I|\leq [a]+1, d(I)\geq a} |y|^{d(I)} \sup_{|z|\leq \beta^{[a]+1} |y|} |X^{I}f(xz) | \ 
		\end{equation}
		would not suffice for many arguments that follow
		and the integral form of the remainder will be used at several places.
		\end{remark}
		
	\subsubsection{Distributions, Convolution and Differential Operators}\label{subsec:distributions_convolutions}
	We define $\mcD(\mbG):= C^\infty_c(\mbG)$ to be the space of compactly supported smooth functions on $\mbG$. The space of distributions on $\mbG$ is given by the dual $\mcD'(\mbG)$ of $\mcD(\mbG)$ with respect to the usual topology and for $\xi \in \mcD'(\mbG)$ and $\phi \in \mcD(\mbG)$ we either write $\langle \xi,\phi\rangle$ or $\xi(\phi)$ for the canonical pairing.
	Given $r \in \mbR$ we introduce the following useful space of test functions
	$$\mfB^{\lambda}_{ r}:= \left\{ \phi\in C^\infty_c (B_\lambda)  \ : \ |X^I \phi | \leq \frac{1}{\lambda^{|\mfs|+d(I)}} \  \forall I : d(I)\leq (r+\max_i\fraks_i)\vee 0  \right\} \ $$
			and write $\mfB_r:=\mfB_r^1$.
	%
	%
	%
	Given $\phi \in \mcD(\mbG)$ we extensively use the following notation
	$$\phi^\lambda_x(z):=\frac{1}{\lambda^{|\mfs|}}\phi \left(\frac{1}{\lambda}\cdot  (x^{-1}z)\right).$$
	We denote by $\|\cdot \|_{L^p}$ the standard $L^p$-norm with respect to the Haar measure, and note that
	%
	%
	by left and right translation invariance of the Haar measure, for $f\in L^1(\mbG)$ 
	$$\int_{\mbG} f(yx)\dd y = \int_{\mbG} f(xy)\dd y = \int_{\mbG}f(y)\dd y = \int_{\mbG}f(y^{-1})\dd y \ ,$$
	cf.\ the discussion after \cite[Thm.~1.1.1]{fischer_ruzhansky_18_quantisation}. 
	%
	%
	One defines the convolution of two functions $\phi, \psi: \mbG\to \mathbb{R}$ as
	\begin{equation}\label{eq:def_convolution}
\psi*\phi(x)= \int \psi(y) \phi(y^{-1}x) dy = \int \psi(xy^{-1}) \phi(y) dy\ .
	\end{equation}
	Note that in general $\psi* \phi(x)\neq \phi * \psi(x)$ but still many of the usual inequalities hold, in particular the Young inequality, cf. \cite[Prop.~1.18]{folland_stein_82_hardy}. One can write
	$$\psi*\phi(x)= \int \psi(y) \phi_y(x) dy $$
	We extend convolution to certain generalised functions, i.e. elements of $\mcD'(\mbG)$, as usual.
	From now on we will use the notation $\tilde{\phi}(z):= \phi(z^{-1})$, one notes that the following identities hold
	\begin{align}
		\langle f, g*\phi \rangle = \langle \tilde{g} * &f, \phi \rangle = \langle   f * \tilde{\phi}, g \rangle \label{eq:throwing_convolution} \\
		\widetilde{f*g} &=  \tilde{g}*\tilde{f} \label{eq:convolution_reorder}\\
		\psi^\lambda*\phi^\lambda = (\psi*\phi)^\lambda\quad &\text{and}\quad (\psi \ast \phi)_x = \psi_x \ast \phi. \notag
	\end{align}
 Since for any right invariant vector field $Y$ and left-invariant vector field $X$ such that $X|_e = Y|_e$ it holds that $\langle X f,g\rangle = -\langle f, X g\rangle $ for $f,\,g \in C^\infty_c(\mbG)$, it follows from the definition of the convolution that
$$ X^I(\psi*\phi) = \psi*( X^I \phi) , \qquad 
		Y^I(\psi*\phi) = (Y^I\psi)* \phi, \qquad 
(X^I\psi)*\phi= \psi*Y^I \phi \ .$$

An important class of differential operators are Rockland operators. 
The following definition is the simplest to state, we refer to \cite[Sec.~4]{fischer_ruzhansky_18_quantisation} for a detailed discussion and the definition in terms of the Rockland condition.
\begin{definition}
We shall call a left-invariant operator $\mathcal{R}$ a Rockland operator if 
it is hypo-elliptic,
and is of homogeneous of order $m$.
It is called positive
if it is formally self-adjoint and
$$
\langle \mathcal Rf,f\rangle_{L^2}\geq0
\qquad\text{for every }f\in C_c^\infty(\mathbb G).
$$
\end{definition}
Let us note that the choice of working with left-invariant operators and requiring homogeneity is purely a matter of convention. The construction and properties of fundamental solutions of such operators are well understood, and largely analogous to those of translations invariant elliptic operators on $\mathbb{R}^d$, cf.\ 
\cite{folland_75_subelliptic, Nice_noteDHZ94, Elst_Robinson}.
\begin{remark}
We recall that sub-Laplacians on Carnot groups are examples with $m=2$. Operators of the form
$$
\mathcal{R}=\sum_{i=1}^d (-1)^{\sigma/\fraks_i} c_i X^{2\sigma/\fraks_i}_i\ ,
$$
where $\sigma\in \mathbb{N}$ is a common multiple of $\{\fraks_i\}_{i=1}^d\in \mathbb{N}$ and $c_i>0$, are examples with $m=2\sigma$.
\end{remark}

	\section{Besov spaces}\label{sec:2}
	We shall work with the following class of weights.
\begin{definition}\label{def:weight}
We say that a function $w:\mathbb{G}\to (0,\infty)$ is a weight if for every $r>0$ there exists $C>0$ for which
\begin{equ}\label{eq:weight}
C^{-1} \leq  \frac{w(x)}{w(y)} \leq C\ 
\end{equ}
uniformly over $x,y\in \mathbb{G}$ satisfying $|xy^{-1}|\wedge |y^{-1}x| \leq r$.
\end{definition}
Note that for any two weights, their products and sums are weights as well.
Recalling that we fixed a homogeneous norm $|\cdot|$ for which 
triangle inequality holds, for $a,\ell\in \mathbb{R}$ by the following are well defined weights
\begin{equ}\label{eq:weights_examle}
p_a(x):=(1+|x|)^a, \qquad e_\ell(x):=\exp \big(\ell(1+|x|)\big)\ .
\end{equ}
\begin{remark}
For later use, let us record the following property of these weights
\begin{equ}\label{weights_trade}
\sup_{x\in\mathbb G}
\frac{e_{\ell+s}(x)p_a(x)}{e_{\ell+t}(x)}
=
\sup_{x\in\mathbb G}(1+|x|)^a
e^{-(t-s)(1+|x|)}
\lesssim_a(t-s)^{-a} \ ,
\end{equ}
which holds uniformly over $s<t$ and $a>0$ in bounded sets.
\end{remark}


		\begin{definition}\label{def:Besov}
		For $p,q\in [1,\infty],$ $\alpha \in \mbR$ we define\footnote{With the usual convention of interpreting the edge cases $p=\infty$ or $q=\infty$ as (essential) suprema.} the space $B^{\alpha,w}_{p,q}(\mbG)$ to consist of those $f\in \mcD'(\mbG)$ such that
			%
			
			%
			%
			\begin{equation}\label{eq:Besov_norm_1}
				\|f\|_{B_{pq}^{\alpha, w}}:=	\inf_{P}\Bigg[\Big(\int_{0}^1\Big\|\sup_{\phi \in \mfB_{-\alpha } }  \frac{ |\langle f- \tilde{P}_x,\phi_x^\lambda\rangle|}{w(x)\lambda^{\alpha}}\Big\|^q_{L^p}  \frac{d\lambda}{\lambda}				\Big)^{1/q} + \Big\|\frac{|{P}_x|_{\mathcal{P}_\alpha}}{w(x)} \Big\|_{L^p}\Bigg] < +\infty \ ,
			\end{equation}	
where the infimum runs over measurable maps $P:\mbG\to \mathcal{P}_\alpha,\ x\mapsto {P}_x$, we write	
			$\tilde{P}_x:= P_x (x^{-1}\cdot)$, and $|\, \cdot\, |_{\mathcal{P}_\alpha}$ denotes any norm on the finite dimensional vector space $\mathcal{P}_\alpha$.
			%
	\end{definition}
	\begin{remark}
	For $\alpha\leq 0$ recall that $\mathcal{P}_\alpha= \{0\}$ and therefore \eqref{eq:Besov_norm_1} actually simplifies to 
	\begin{equation}\label{eq:Negative_Holder_Def}
				\|f\|_{B_{pq}^{\alpha, w}}=	\Big(\int_{0}^1 \Big\| \sup_{\phi \in \mfB_{ -\alpha }} \frac{|\langle f,\phi^\lambda_x\rangle|}{w(x)\lambda^\alpha} \Big\|^q_{L^p} \frac{d\lambda}{\lambda}			\Big)^{1/q}	
				< +\infty \ .
			\end{equation}
	\end{remark}
	\begin{remark}\label{rm:changing_scales}
	Note that replacing $\lambda$ by $C\lambda$ for some $C>0$ in \eqref{eq:Besov_norm_1} gives equivalent norms. (The argument for $C<1$ is straightforward, while for $C>1$ one writes the test function $\phi^{C\lambda}$ as a sum of a fixed number of other functions in $\mathfrak{B}^\lambda$, cf. \cite[Rem.~4.10]{HS23manifolds}).
	
	Similarly, one straightforwardly sees that for any $C>0$
	\begin{equation}\label{eq:Besov_norm_1'}
			\inf_{P}	\Bigg[\Big(\int_{0}^C\Big\|\sup_{\phi \in \mfB_{-\alpha  }}  \frac{ |\langle f- \tilde{P}_x,\phi_x^\lambda\rangle|}{w(x)\lambda^{\alpha}}\Big\|^q_{L^p}  \frac{d\lambda}{\lambda}				\Big)^{1/q} + \Big\|\frac{|{P}_x|_{\mathcal{P}_\alpha}}{w(x)} \Big\|_{L^p}\Bigg] \ ,
			\end{equation}	
	as well as for any $\mfr>1$, $n_0\in \mathbb{Z}$ 
	\begin{equation}\label{eq:Besov_norm_1''}
\inf_{P}	\Bigg[\Big(\sum_{n=n_0}^\infty \Big\|\sup_{\phi \in \mfB_{-\alpha } }  \frac{ |\langle f- \tilde{P}_x,\phi_x^{\mfr^{-n}}\rangle|}{w(x)\mfr^{-n\alpha}}\Big\|^q_{L^p} 				\Big)^{1/q} + \Big\|\frac{|{P}_x|_{\mathcal{P}_\alpha}}{w(x)} \Big\|_{L^p}\Bigg] \ ,
	\end{equation}	
are equivalent norms to \eqref{eq:Besov_norm_1}.
	\end{remark}


\begin{remark}
Note that our conventions are to measure regularity using
left-invariant vector fields and the displacement $x^{-1}y$. Replacing
these by right-invariant vector fields and $yx^{-1}$ gives 
corresponding right Besov spaces. The inversion map
$f(x)\mapsto f(x^{-1})$ identifies the two conventions, with the
weight $w$ replaced by $x\mapsto w(x^{-1})$.
\end{remark}

	\begin{lemma}\label{prop:holder_inclusion_compact}
			 If $\alpha>0$,
			then for any distribution $f \in B^{\alpha,w}_{p,q}(\mbG)$ it holds that $X^{I}f/w\in L^{p}$ for every $d(I)<\alpha$. Furthermore, there exists a map $P: \mbG\to \mathcal{P}_\alpha,\ x\mapsto {P}_x$ for which the infimum in \eqref{eq:Besov_norm_1} is obtained, and any such map agrees with the Taylor polynomial of $f$ a.e.
	\end{lemma}
	%
\begin{proof}
Fix $\phi\in \mfB_{- \alpha}$ such that $\int \phi>0.$ The lemma follows from the inequality
\begin{equ}\label{eq:technical_taylor_id}
\Big\|\Big\langle \frac{X^I f- X^I P_x(e)}{w(x)}, \phi^\lambda_x\Big\rangle\Big\|_{L^p} \lesssim \Big\|\frac{ |\langle f-  \tilde{P}_x, X^I \phi^\lambda_x\rangle|}{w(x)}\Big\|_{L^p}
+\Big\| \frac{|{P}_x|_{\mathcal{P}_\alpha}}{w(x)}\Big\|_{L^p} \lambda\ .
\end{equ}
Indeed both terms vanish (at least) along a subsequence $\lambda\to 0$ and therefore 
 $X^{I}f= X^I P_x(e)$ a.e.\ for $d(I)<\alpha$,
To see \eqref{eq:technical_taylor_id} write
$$\langle X^I f- X^I P_x(e), \phi^\lambda_x\rangle=  \langle X^I f- X^I \tilde{P}_x, \phi^\lambda_x\rangle +\langle X^I \tilde{P}_x- X^I \tilde{P}_x(x), \phi^\lambda_x\rangle \ . 
$$
Integrating this and noting that
$$\Big\|\Big\langle\frac{X^I \tilde{P}_x- X^I \tilde{P}_x(x)}{w(x)}, \phi^\lambda_x\Big\rangle\Big\|_{L^p}\lesssim \Big\| \frac{|{P}_x|_{\mathcal{P}_\alpha}}{w(x)}\Big\|_{L^p} \lambda $$
since $
 |\langle X^I {P}_x- X^I {P}_x(e), \phi^\lambda\rangle| \lesssim | {P}_x|_{\mathcal{P}_\alpha} \lambda
$
completes the proof.
%

\end{proof}	
%
%
%
%
%
	\begin{lemma}
For any $\alpha\in \mathbb{R}$ and any multi-index $I$, the linear map $X^I:\  B^{\alpha,w}_{pq}(\mbG)\to B^{\alpha-d(I),w}_{pq} (\mbG), \  f\mapsto X^I f$ is bounded.
\end{lemma}
\begin{proof}
We first consider the case $\alpha-d(I)\leq 0$. Here it suffices to 
note that for every $\alpha, I$ there exists $c_{\alpha,I}>0$ such that $c_{\alpha,I}X^{I}\phi \in  \mathfrak{B}_{-\alpha}$ whenever $\phi \in  \mathfrak{B}_{-\alpha+d(I)}$.

When $\alpha-d(I)>0$ one first obtains the desired bound uniformly over more regular functions as above. Then, the full estimate follows using Lemma~\ref{prop:holder_inclusion_compact}.

\end{proof}

\subsection{A Wavelet-like Characterisation}
Throughout this section we fix $\mfr>1$ sufficiently large.  
\begin{assumption}\label{ass:pair}
For given $R\geq 0$
let $(\varphi,\rho)$ be a pair of functions belonging to $C_c^\infty\big(B_1(e)\big)$ with the following properties
\begin{enumerate}[before=\vspace{3pt}, after=\vspace{4pt}]
\item\label{it1} $\int \rho(x)\dd x = 1$ and $\int \eta^I(x) \rho(x)\dd x = 0$ for every $I$ such that $0<d(I)\leq R$,
\item\label{Item3} for $n<m$ it holds that $\rho^{(n,m)} := \rho^{(n)}\ast \rho^{(n+1)}\ast \cdots \ast \rho^{(m)}\to \varphi^{(n)}$ in $\mathcal{D}$ as $m\to \infty$ ,
\end{enumerate}
where  $\rho^{(n)}(x):= \mfr^{n|\mfs|}\rho (\mfr^{n}\cdot x)$ and $\varphi^{(n)}(x):= \mfr^{n|\mfs|}\varphi (\mfr^{n}\cdot x)$. 
\end{assumption}

\begin{remark}\label{rem:existence}
The existence of such pairs for $\mfr=\mfr(\mathbb{G})$ sufficiently large is elementary: 
one constructs $\rho$ such that Item~\ref{it1} holds by basic linear algebra. Then, one sees that Item~\ref{Item3} holds whenever
$\|\rho\|_{L^1}< \mfr$. For details, we refer to \cite[Lem.~3.10]{MS25} where such a pair was used to prove the reconstruction theorem by adapting arguments from $\mathbb{R}^d$ in \cite{friz_hairer_20_introduction}.
\end{remark}

\begin{theorem}\label{prop:equiv}
Let $p,q\in [1,\infty]$, $\alpha\in \mathbb{R}\setminus \triangle$. Let $(\varphi,\rho)$ be as in Assumption~\ref{ass:pair} for $R>|\alpha|$.
\begin{enumerate}
\item
For $\xi\in B^{\alpha,w}_{p,q}$ an equivalent norm to the one in Definition~\ref{def:Besov} is given as follows:
\begin{itemize}
\item for $\alpha<0$
\begin{equ}\label{equiv}
	\Big(\sum_{n\geq 0} \mfr^{n q\alpha} \Big\| \frac{\xi_{n+1}}{w}\Big\|^q_{L^p}\Big)^{1/q}\ ,
\end{equ}
\item for $\alpha>0$
\begin{equ}\label{equiv2}
	\Big\| \frac{\xi_0}{w}\Big\|_{L^p} 
+\max_{\substack{|K|\leq [\alpha]+1, \\ d(K)>\alpha }}\Big(\sum_{n\in \mathbb{N}} \mfr^{nq(\alpha-d(K))} \Big\| \frac{X^{K}\xi_{n}}{w}\Big\|^q_{L^p}\Big)^{1/q} \ ,
\end{equ}
\end{itemize}
where we have set $\xi_n= \xi* \tilde{\varphi}^{(n)}$.
\item\label{item2} Assume that a sequence of functions $\xi_n :\mbG\rightarrow \mbR$  is related by $\xi_n =  \xi_{n+1}\ast \tilde{\rho}^{(n)}$ and that 
\eqref{equiv}, resp. \eqref{equiv2}  is finite. Then, there exists $\xi \in B^{\alpha,w}_{pq}$ such that $\xi_n =\xi\ast \tilde{\varphi}^{(n)}$ and
 for any $\bar{\alpha}<\alpha$
$$\xi_n \to \xi \qquad \text{ in } B^{\bar{\alpha},w}_{pq} \ .$$
%
\end{enumerate}
\end{theorem}		

\begin{remark}\label{rem:wavelet_higher}
It will be clear from the proof, that one obtains equivalent norms replacing the role of $0$ respectively $1$ in \eqref{equiv} and \eqref{equiv2} by any $n_0\in \mathbb{N}$, resp.\ $n_0+1$.
\end{remark}
%

To improve exposition, we shall present the proof of Theorem~\ref{prop:equiv} separately in the cases $\alpha<0$ and $\alpha>0$. First,
define for $\gamma>0$ and any compactly supported sequence $a_{m}: \mathbb{N}\to \mathbb{R}$
\begin{equ}\label{operator}
T_\gamma\big[\{a_m\}_m\big](\lambda):= \sum_{\substack{m\in \mathbb{N}\\ \mfr^{-m}\leq \lambda}} a_m \frac{\mfr^{-m\gamma}}{\lambda^\gamma} \ .
\end{equ}
\begin{lemma}\label{lem:operatorbound}
For every $q\in [1,\infty]$, the map defined in \eqref{operator} extends to a bounded linear map $T_\gamma: \ell^{q}(\mathbb{N})\to L^{q}((0,1), \frac{d\lambda}{\lambda})$.
\end{lemma}
\begin{proof}
By interpolation, it suffices to check the cases $q=1$ and $q=\infty$. For $q=1$, the claim follows from Fubini since 
$$\sup_{m\in \mathbb{N}}\Big\|\mathbf{1}_{\mfr^{-m}\leq \lambda} \frac{\mfr^{-m\gamma}}{\lambda^\gamma} \Big\|_{L^1((0,1),\frac{d\lambda}{\lambda})}<\infty \ .$$
For $q= \infty$ one notes that 
$$\sup_{\lambda\in (0,1)} \Big[\sum_{m: \mfr^{-m}<\lambda} \frac{\mfr^{-m\gamma}}{\lambda^\gamma}\Big]<\infty \ .$$
\end{proof}

	\begin{proof}[Proof of Theorem~\ref{prop:equiv} for $\alpha\leq 0$]
	For $\xi\in B^{\alpha,w}_{p,q}$, note that 
	$\xi_n(x)=\xi*\tilde{\varphi}^{(n) }= \langle \xi, \varphi^{n}_x\rangle .$	
	Since there exists a constant $c>0$ such that $c\varphi^{(n)}\in \mathfrak{B}^{\mfr^{-n}}_{r}$, it follows by Remark~\ref{rm:changing_scales} that the quantity in \eqref{equiv} can be bounded by the Besov norm of $\xi$.

Next we check Item~\ref{item2}, the full theorem will follow from the quantitative estimates established in doing so. Assume that $q<\infty$ (the case $q=\infty$ is simpler and left to the reader). Let $-R<\alpha'<\alpha$.
Note that 
$
\langle \xi_n-\xi_{n+1}, \phi^\lambda_x\rangle
=	 \langle \xi_{n+1}, (\phi^\lambda_x* \rho^{(n)}  -\phi^\lambda_x) \rangle \ . 
$ Therefore,
\begin{equs}
\Big\| \sup_{\phi \in \mfB_{ -\alpha' }} \frac{|\langle \xi_n-\xi_{n+1},\phi^\lambda_x\rangle|}{w(x)\lambda^{\alpha'}} \Big\|_{L^p} 	
&\lesssim \Big\|\frac{\xi_{n+1}}{w }\Big\|_{L^p}  \sup_{\phi \in \mfB_{ -{\alpha'}}} \frac{\|\phi^\lambda_x* \rho^{(n)}  -\phi^\lambda_x \|_{L^1}}{\lambda^{{\alpha'}}} 		&\\
	& \lesssim \Big\|\frac{\xi_{n+1}}{w }\Big\|_{L^p} \lambda^{-{\alpha'}}\ , \label{eq:small scale}
\end{equs}
%
 where the first inequality uses that $w$ is a weight and Young's inequality. 
Noting that by Assumption~\ref{ass:pair} 
$$\phi^\lambda_x* \rho^{(n)}(y)  -\phi^\lambda_x(y)= (\phi^\lambda_x -\tilde{\opP}_{y}^{R}[\phi^\lambda_x])* \rho^{(n)}(y) , $$
$\|\phi^\lambda_x* \rho^{(n)}(y)  -\phi^\lambda_x(y)\|_{L^1} \lesssim \lambda^{-R}\mfr^{-nR}$ and therefore
\begin{equs}
	 \Big\| \sup_{\phi \in \mfB_{ -{\alpha'} \rceil }} \frac{|\langle \xi_n-\xi_{n+1},\phi^\lambda_x\rangle|}{w(x)\lambda^{\alpha'}} \Big\|_{L^p} 	&\lesssim \Big\|\frac{\xi_{n+1}}{w }\Big\|_{L^p} \sup_{\phi \in \mfB_{ -{\alpha'} }} \frac{\|\phi^\lambda_x* \rho^{(n)}  -\phi^\lambda_x) \|_{L^1}}{\lambda^{{\alpha'}}} 	&\\
	& \lesssim \Big\|\frac{\xi_{n+1}}{w }\Big\|_{L^p} \lambda^{-R-{\alpha'}}\mfr^{-nR}\ . \label{eq:large scale}
\end{equs}
Taking the $L^{q}\big(\frac{d\lambda}{\lambda}\big)$ norm of  \eqref{eq:small scale} over $(0,\mfr^{-n})$ and of \eqref{eq:large scale} over $(\mfr^{{-n}},1)$ we find that $\|\xi_{n+1}-\xi_{n}\|_{B^{{\alpha'},w}_{pq}}\lesssim \Big\|\frac{\xi_{n+1}}{w }\Big\|_{L^p} \mfr^{n{\alpha'}}$, 
and thus for $\alpha'=\bar{\alpha}<\alpha$
$$
\sum_{n\geq 0} \|\xi_{n+1}-\xi_{n}\|_{B^{\bar{\alpha},w}_{pq}}\leq \sum_{n\geq 0} \Big\|\frac{\xi_{n+1}}{w }\Big\|_{L^p} \mfr^{n\bar{\alpha}} <\infty \ .
$$

We conclude that $\xi_n \to \xi$ in $B_{pq}^{\bar{\alpha},w}
$, since clearly $\xi_0\in B_{pq}^{{\alpha},w}$.
Since $\xi_n= \xi_m* \tilde{\rho}^{(n,m-1)}$, taking $m\to \infty$ we obtain $\xi_n= \xi*\tilde{\varphi}^{(n) }$.

It remains to check that $\|\xi\|_{B^{\alpha,w}_{pq}} $ is bounded by \eqref{equiv}. For $\lambda\in (0,1)$ set $n_\lambda:= \min \{n\in \mathbb{N}: \mfr^{-n}<\lambda\}$
and write  
\begin{equ}\label{eq:sum_dec}
\langle\xi, \phi^\lambda_{x}\rangle= \langle\xi_{n_\lambda}, \phi^\lambda_{x}\rangle + \sum_{m\geq n_\lambda}\langle\xi_{m+1}-\xi_{m} , \phi^\lambda_{x}\rangle\ ,
\end{equ}
and therefore
\begin{align*}
\| \xi \|_{B^{\alpha,w}_{pq}}
&\leq \Big(\int_{0}^{1} \Big( \Big\| \sup_{\phi \in \mfB_{ -\alpha }} \frac{|\langle \xi_{n_\lambda},\phi^\lambda_x\rangle|}{w(x)\lambda^\alpha} \Big\|_{L^p}+\sum_{m\geq n_\lambda}\Big\|\sup_{\phi \in \mfB_{ -\alpha }} \frac{|\langle \xi_{m+1}-\xi_{m},\phi^\lambda_x\rangle|}{w(x)\lambda^\alpha} \Big\|_{L^p}  \Big)^q \frac{d\lambda}{\lambda}	\Big)^{1/q}	\\
&\lesssim \Big(\int_{0}^{1} \Big( \Big\|\frac{\xi_{n_{\lambda}+1}}{w }\Big\|_{L^p} \lambda^{\alpha} \frac{d\lambda}{\lambda}	\Big)^{1/q}	
+\Big(\int_{0}^{1}\sum_{m\geq n_\lambda}\Big\|\frac{\xi_{m+1}}{w }\Big\|_{L^p} \lambda^{-R-\alpha}\mfr^{-mR}  \Big)^q \frac{d\lambda}{\lambda}	\Big)^{1/q} \ .
\end{align*}
Since $\lambda^{-\alpha}\sim\mfr^{ n_\lambda \alpha}$ the first term is bounded by \eqref{equiv}. 
For the second term, set $a_m:=\|\xi_{m+1}/w\|_{L^p}\mfr^{m\alpha} $ and let $T_{\alpha+R}$ be the operator defined in \eqref{operator}. Then,
\begin{equs}
&\Big(\int_{0}^{1}\Big(\sum_{m\geq n_\lambda}\Big\|\frac{\xi_{m+1}}{w }\Big\|_{L^p} \lambda^{-R-\alpha}\mfr^{-mR}  \Big)^q \frac{d\lambda}{\lambda}	\Big)^{1/q}
&=\Big(\int_{0}^{1}\Big(\sum_{m\geq n_\lambda}\Big\|\frac{\xi_{m+1}}{w }\Big\|_{L^p}\mfr^{m\alpha} \frac{\mfr^{-m(R+\alpha)}}{\lambda^{R+\alpha}}  \Big)^q \frac{d\lambda}{\lambda}	\Big)^{1/q}\\
&= \Big(\int_{0}^{1}\Big(
T_{R+\alpha}\big[ \big\{ a_m\big\}_m \big](\lambda)\Big)^q
\frac{d\lambda}{\lambda}	\Big)^{1/q}\ 
\end{equs}
which is bounded by \eqref{equiv} by Lemma~\ref{lem:operatorbound}.
\end{proof}

\begin{proof}[Proof of Theorem~\ref{prop:equiv} for $\alpha>0$]

The claim that whenever $\xi\in B^{\alpha,w}_{p,q}$ the quantity in \eqref{equiv2} is bounded by a multiple of $\|\xi\|_{B^{\alpha,w}_{p,q}}$ is seen very similarly to the case $\alpha<0$, by using that 
$\xi_{0}(x)= \langle \xi, \varphi_{x}\rangle $ and that
$$ X^{K}\xi_n(x)=X^{K} \int \xi(y) \varphi^{(n)}(x^{-1}y)dy=\langle\xi,  (Y^{K}\varphi^{(n)})_x\rangle= \langle \xi(y)-\tilde{P}_x , (Y^{K}\varphi^{(n)})_x\rangle
$$
where in the first and second equality we used Lemma~\ref{prop:holder_inclusion_compact}.
 Thus we turn to the proof of Item~\ref{item2}. We observe that by Theorem~\ref{th:taylor} for any smooth function and $\alpha'>0$
\begin{equ}\label{eq:taylor_expansion_formula}
\langle f- \tilde{\opP}^{\alpha'}_x[f], \phi^\lambda_x \rangle
= \int_{\mathbb{G}} \big(f(xy)- {\opP}^{\alpha'}_x[f](y)\big) \phi^\lambda (y) dy 
=  \sum_{\substack{|I|\leq [\alpha']+1, \\ d(I)>{\alpha'} }}\int_{\mathbb{G}} \int X^{I}f(xz) Q^{I}(y,dz) \phi^\lambda (y) dy ,
\end{equ}
and therefore
\begin{equ}\label{eq:techical>0est}
\Big\| \frac{\langle f- \tilde{\opP}^{\alpha'}_x[f], \phi^\lambda_x \rangle}{w(x)} \Big\|_{L^p_x} \lesssim \sum_{\substack{|I|\leq [\alpha']+1, \\ d(I)>{\alpha'} }} \int_{\mathbb{G}} \int_{\mathbb{G}}  \Big\| \frac{X^{I}f}{w} \Big\|_{L^p} |Q^{I}|(y,dz) \phi^\lambda (y) dy 
\lesssim \sum_{\substack{|I|\leq [\alpha']+1, \\ d(I)>{\alpha'} }} \lambda^{d(I)} \Big\| \frac{X^{I}f}{w} \Big\|_{L^p} \ .
\end{equ}
Applying this to $\xi_0$ one finds that
$$\Big(\int_0^1 \Big\| \frac{\langle \xi_0- \tilde{\opP}^{\alpha'}_x[\xi_0], \phi_x^\lambda \rangle}{w(x)\lambda^{\alpha'}} \Big\|_{L^p_x}^q \frac{d\lambda}{\lambda}\Big)^{1/q}
\lesssim \sum_{\substack{|I|\leq [\alpha']+1, \\ d(I)>{\alpha'} }}  \Big\| \frac{X^{I}\xi_0}{w} \Big\|_{L^p}\ .
 $$

Next we turn our attention to $\xi_{n+1}-\xi_n$ and
 find that 
\begin{equs}
\Big(\int_0^{\mfr^{-n}} \Big\| \frac{\langle \xi_{n+1}-\xi_n- \tilde{\opP}^{\alpha'}_x[\xi_{n+1}-\xi_n], \phi^\lambda_x \rangle}{w(x)\lambda^{\alpha'}} \Big\|_{L^p_x}^q \frac{d\lambda}{\lambda}\Big)^{1/q}
&\lesssim \sum_{\substack{|I|\leq [\alpha']+1, \\ d(I)>{\alpha'} }} \Big( \int_{0}^{\mfr^{-n}} \lambda^{(d(I)-{\alpha'})q} \frac{d\lambda}{\lambda}\Big)^{1/q} \Big\| \frac{X^{I}[\xi_{n+1}-\xi_n]}{w} \Big\|_{L^p} \\
&\lesssim  \sum_{\substack{|I|\leq [\alpha']+1, \\ d(I)>{\alpha'} }} \mfr^{-n(d(I)-{\alpha'})}  \Big\| \frac{X^{I}[\xi_{n+1}-\xi_n]}{w} \Big\|_{L^p} \\
&\lesssim  \sum_{\substack{|I|\leq [\alpha']+1, \\ d(I)>{\alpha'} }} \mfr^{-n(d(I)-{\alpha'})}  \Big\| \frac{X^{I}\xi_{n+1}}{w} \Big\|_{L^p} \ . \label{eq:alpha>0}
\end{equs}
Similarly, 
\begin{equ}\label{eq:alpha>0a}
 \Big\| \frac{\langle \xi_{n+1}-\xi_n, \phi^\lambda_x \rangle}{w(x)\lambda^{\alpha'}} \Big\|_{L^p_x}
\lesssim \Big\| \frac{\xi_{n+1}-\xi_n}{w} \Big\|_{L^p}  \lambda^{-{\alpha'}}  
\lesssim \lambda^{-{\alpha'}}   \sum_{\substack{|K|\leq [{\alpha'}]+1, \\ d(K)>{\alpha'}}} \Big\|\frac{X^K \xi_{n+1}}{w} \Big\|_{L^p} \mfr^{-nd(K)} 
\end{equ}
and 
\begin{equs}
 \Big\| \frac{\langle \tilde{\opP}^{\alpha'}_x[\xi_{n+1}-\xi_n], \phi_x^\lambda \rangle}{w(x)\lambda^{\alpha'}} \Big\|_{L^p_x}
&\lesssim
 \sum_{d(I)< {\alpha'}}  \Big\| \frac{X^{I}[\xi_{n+1}-\xi_{n}]}{w} \Big\|_{L^p}  \lambda^{-{\alpha'}+d(I)} \\
&\lesssim 
\sum_{\substack{|K|\leq [{\alpha'}]+1, \\ d(K)>{\alpha'}}} \Big\|\frac{X^K \xi_{n+1}}{w} \Big\|_{L^p} \mfr^{-nd(K)} 
 \sum_{d(I)< {\alpha'}}   \frac{\lambda^{-{\alpha'}+d(I)}}{\mfr^{-n d(I)}} \ , \label{eq:alpha>0b}
\end{equs}
where for the first inequality one notes that the coefficient of $\eta^{I}$ in $\tilde{\opP}^{\alpha'}[f]$ is a linear combination of $\{X^{J}f\}_{d(J)\leq d(I)}$, and therefore for some $c_{I,J}\in \mathbb{R}$ (only depending on $\mathbb{G}$)
$$\langle \tilde{\opP}^{\alpha'}_x[f], \phi_x^\lambda\rangle = \sum_{d(I)< {\alpha'}} \sum_{d(J)\leq d(I)} c_{IJ} X^Jf(x) \langle \eta^I, \phi^\lambda\rangle \ .$$
For the second inequality in \eqref{eq:alpha>0a} and \eqref{eq:alpha>0b} we used that for $d(I)<\alpha'$
\begin{equ}\label{eq:bounding_by_high_derivative}
\Big\| \frac{X^{I}[\xi_{n+1}-\xi_{n}]}{w}\Big\|_{L^p} \lesssim 
\mfr^{nd(I)} \max_{\substack{|K|\leq [{\alpha'}]+1, \\ d(K)>{\alpha'}}} \Big\|\frac{X^K \xi_{n+1}}{w} \Big\|_{L^p} \mfr^{-nd(K)} \ ,
\end{equ}
which in turn follows since
\begin{equs}
X^{I}[\xi_{n+1}-\xi_n](x)&=  \int_{\mathbb{G}} \big(X^{I}\xi_{n+1}(x)-X^{I}\xi_{n+1}(y)\big)\rho^{n}_x(y) dy \\
&= -\int_{\mathbb{G}} \big(X^{I}\xi_{n+1}(y)-\tilde{\opP}^{{\alpha'}-d(I)}_{y}[X^{I}\xi_{n+1}](x)\big)\rho^{n}_x(y) dy\\
&=
-\sum_{\substack{|J|\leq [{\alpha'}-d(I)]+1, \\ d(J)>{\alpha'}-d(I) }}
 \int_{\mathbb{G}}\int_{\mathbb{G}} X^{J}X^{I}\xi_{n+1}(xz) Q^{J}(y,dz)\rho^{n}_x(y) dy \ .
\end{equs}
Thus, for $\alpha'=\bar{\alpha}<\alpha$, by \eqref{eq:alpha>0} and taking the $q$-th power of \eqref{eq:alpha>0a} and \eqref{eq:alpha>0b} and integrating over $\lambda\in (\mfr^{-n},1)$ we find that
 $$\sum_n\|\xi_{n+1}-\xi_n\|_{B^{\bar{\alpha},w}_{pq}}<\infty \ . $$
  The fact that the limit $\xi\in B^{\bar{\alpha},w}_{pq}$ satisfies $\xi_n= \xi*\tilde{\varphi}^{(n)}$ follows as for $\alpha<0$.

We next check that $\xi\in B^{{\alpha},w}_{pq}$. 
For $\lambda\in (0,1)$ we set $n_\lambda:= \min \{n\in \mathbb{N}: \mfr^{-n}<\lambda\}$ and write  
\begin{equ}\label{eq:sum_dec2}
\langle\xi-\tilde{\opP}_x[\xi], \phi^\lambda_{x}\rangle= \langle\xi_{n_\lambda}-\tilde{\opP}_x[\xi_{n_\lambda}], \phi^\lambda_{x}\rangle 
+ \sum_{m\geq n_\lambda}\langle \xi_{m+1}-\xi_m- \tilde{\opP}_x[\xi_{m+1}-\xi_m] , \phi^\lambda_{x}\rangle\ ,
\end{equ}
and therefore
\begin{equs}
\| \xi \|_{B^{\alpha,w}_{pq}}
&\leq \Big(\int_{0}^{1} \Big( \Big\| \sup_{\phi \in \mfB_{ -\alpha }} \frac{|\langle \xi_{n_\lambda}-\tilde{\opP}_x[\xi_{n_\lambda}],\phi^\lambda_x\rangle|}{w(x)\lambda^\alpha} \Big\|_{L^p}+\sum_{m\geq n_\lambda}\Big\| \frac{|\langle  \xi_{m+1}-\xi_m- \tilde{\opP}_x[\xi_{m+1}-\xi_m],\phi^\lambda_x\rangle|}{w(x)\lambda^\alpha} \Big\|_{L^p}  \Big)^q \frac{d\lambda}{\lambda}	\Big)^{1/q}	\\
&\lesssim \sum_{\substack{|I|\leq [\alpha]+1, \\ d(I)>{\alpha} }} 
\Big(\int_{0}^{1} \Big(   \Big\| \frac{X^{I}\xi_{n_{\lambda}}}{w} \Big\|_{L^p}  \lambda^{-{(\alpha-d(I))}} \Big)^{q}\frac{d\lambda}{\lambda}	\Big)^{1/q}\label{eq:second_to_last_line}	\\
&\qquad+\sum_{\substack{|K|\leq [{\alpha}]+1, \\ d(K)>{\alpha}}}\sum_{d(I)< {\alpha}} \Big(\int_{0}^{1} \Big(\sum_{m\geq n_\lambda} \Big\|\frac{X^K \xi_{m+1}}{w} \Big\|_{L^p} \mfr^{-md(K)} 
  \frac{\lambda^{-{\alpha}+d(I)}}{\mfr^{-m d(I)}} \Big)^q \frac{d\lambda}{\lambda}	\Big)^{1/q}\ . \label{last_line}\\
\end{equs}
where in the second inequality to bound the first term we used \eqref{eq:techical>0est}, and for the second one we used \eqref{eq:alpha>0a} and \eqref{eq:alpha>0b}.
To conclude, note that as in the case $\alpha<0$ the term \eqref{eq:second_to_last_line} is bounded by \eqref{equiv2} since $\lambda^{-(\alpha-d(I))}\sim\mfr^{ n_\lambda (\alpha-d(I))}$. Finally, for \eqref{last_line}, note that for each $K, I$ as in that summand is equal to 
$$ \Big(\int_{0}^{1} \Big(\sum_{m\geq n_\lambda} \Big\|\frac{X^K \xi_{m+1}}{w} \Big\|_{L^p} \mfr^{m(\alpha-d(K))} 
  \frac{\mfr^{-m(\alpha- d(I))}}{\lambda^{{\alpha}-d(I)}} \Big)^q \frac{d\lambda}{\lambda}	\Big)^{1/q}
= \Big(\int_{0}^{1}\Big(
T_{\alpha-d(I)}\big[ \big\{ a^K_m\big\}_m \big](\lambda)\Big)^q
\frac{d\lambda}{\lambda}	\Big)^{1/q}\   
  $$
for $a^K_m:=\Big\|\frac{X^K \xi_{m+1}}{w} \Big\|_{L^p} \mfr^{m(\alpha-d(K))}  $ and is thus bounded by \eqref{equiv2} by Lemma~\ref{lem:operatorbound}.

Finally, it follows straightforwardly from \eqref{eq:bounding_by_high_derivative} that $X^{I}\xi_n/w$ converges in $L^{p}(\mathbb{G})$ and the limit satisfies the desired bound.
\end{proof}
%
\subsection{Embeddings}

One also directly recovers the following Besov embeddings.
\begin{corollary}\label{prop:embedding}
The identity map is a continuous embedding for the following spaces:	
		%
		\begin{enumerate}
			\item
		 $B^{\alpha,w}_{pq}(\mbG)\subset B^{\beta,w}_{pq}(\mbG)$ for every $\beta<\alpha$ ,
		 \item\label{prop_item:embedding}
			$B^{\alpha,w}_{pq}(\mbG)\subset B^{\alpha,w}_{pq'}(\mbG)$ for every $q<q'$,
			\item $
B^{\alpha,w'}_{p,q}(\mathbb G)
   \subset B^{\alpha,w}_{p,q}(\mathbb G)$ for weights
 $w,w'$  such that 
$
\sup_{x\in\mathbb G}\frac{w'(x)}{w(x)}<\infty,
$
\item\label{Item_with_num_cond}  $B^{\alpha_1,w}_{p_1,q}\subset B^{\alpha_0,w}_{p_0,q}$ for $\alpha_0,\alpha_1\in \mathbb{R}\setminus\triangle$ whenever 
$
\alpha_1-\alpha_0\geq\frac{|\fraks|}{p_1}-\frac{|\fraks|}{p_0}>0\ .
$
\end{enumerate}
\end{corollary}
\begin{proof}
The first two items are immediate from Definition~\ref{def:Besov}. 
The third is a consequence of the monotonicity property 
$$\|f\|_{B^{\alpha,w}_{p,q}} \leq \sup_{x\in \mathbb{G}} \Big(\frac{w'(x)}{w(x)}\Big) \|f\|_{B^{\alpha,w'}_{p,q}} \ ,$$
with respect to weights $w,w'$.

For the proof of Item~\ref{Item_with_num_cond}, we
 first consider the case $\alpha_1<0$.
Let $\frac{1}{p_0}+1 = \frac{1}{p_1}+ \frac{1}{p_2}$, then
\begin{equs}\label{ineq}
\Big\|\frac{f_n}{w}\Big\|_{L^{p_0}}&= \Big\|\frac{f_{n+1}*\tilde{\rho}^{(n)}}{w}\Big\|_{L^{p_0}} 
\leq \|\rho^{(n)}\|_{L^{p_2}} \Big\|\frac{f_{n+1}}{w}\Big\|_{L^{p_1}} \lesssim \mfr^{n(|\fraks|- |\fraks|/p_2) }\Big\|\frac{f_{n+1}}{w}\Big\|_{L^{p_1}} \ .
\end{equs}
Thus,
$\mfr^{n\alpha_0} \Big\|\frac{f_n}{w}\Big\|_{L^{p_0}}= \mfr^{n(|\fraks|- |\fraks|/p_2 +\alpha_0-\alpha_1) } \mfr^{n\alpha_1} \Big\|\frac{f_{n+1}}{w}\Big\|_{L^{p_1}} $ and we note that $|\fraks|- |\fraks|/p_2 +\alpha_0-\alpha_1\leq 0$ is equivalent to the numerical condition in Item~\ref{Item_with_num_cond}.

If $\alpha_1>0$, one additionally notes that $X^I f_n= f_{n+1}* X^I \tilde{\rho}^{(n)}$, and $\| X^I \tilde{\rho}^{(n)}\|_{L^{p_2}}= \mfr^{nd(I)}\| \tilde{\rho}^{(n)}\|_{L^{p_2}}$ and argues as above. 
\end{proof}

\subsection{Characterisation by Taylor remainders}

\begin{theorem}\label{prop:difference}
For $\alpha\in \mathbb{R}_+\setminus \triangle$ an equivalent norm to the one in Definition~\ref{def:Besov} is given for $f\in B^{\alpha,w}_{pq}$ by 
\begin{equ}\label{eq:difference}
\sum_{d(I)<\alpha }\Big\| \frac{X^I f}{w}\Big\|_{L^p} + \Big(\int_{|h|<1} \Big\| \frac{f(xh)-\opP^\alpha_x[f](h)}{w(x) |h|^\alpha}\Big\|^q_{L^p_x} \frac{dh}{|h|^{|\fraks|}}\Big)^{1/q} \ .
\end{equ}
\end{theorem}


Before turning to the proof, we define for $\gamma>0$ the following maps. For a compactly supported sequence $a_{m}: \mathbb{N}\to \mathbb{R}$, respectively 
$f\in C_c(B_{1}\setminus\{0\})$
\begin{equ}\label{operator3}
\tilde{T}_\gamma\big[\{a_m\}_m\big](h):= \sum_{\substack{m\in \mathbb{N}:\\\mfr^{-m}\leq |h|}} a_m \frac{\mfr^{-m\gamma}}{|h|^\gamma} \ ,
\qquad
{T}'_\gamma[f](\lambda):= \int f(h) \frac{|h|^\gamma}{\lambda^\gamma} \sup_{\substack{\phi\in C_{c}(B_1)\\ \| \phi\|_{L^\infty}\leq 1}} |\phi^{\lambda}(h)| dh \ .
\end{equ}
The proof of the following lemma is essentially the same as the proof of Lemma~\ref{lem:operatorbound} by interpolating the cases $q=1$ and $q=\infty$.
\begin{lemma}\label{lem:operatorbound2}\label{lem:operatorbound3}
For every $q\in [1,\infty]$, the maps defined in \eqref{operator3} extend to bounded linear maps 
$$\tilde{T}_\gamma: \ell^{q}(\mathbb{N})\to L^{q}(B_{1}, \frac{dh}{|h|^{|\fraks|}}), \qquad \text{resp.}\qquad {T}'_\gamma: L^{q}(B_{1}, \frac{dh}{|h|^{|\fraks|}})\to L^{q}((0,1), \frac{d\lambda}{\lambda})\ .$$
\end{lemma}

\begin{proof}[Proof of Theorem~\ref{prop:difference}]
We first show that \eqref{eq:difference} is bounded by the Besov norm. That the first summand is bounded, follows directly from Lemma~\ref{prop:holder_inclusion_compact}. To treat the second term, for $h\in B_1$
we write $m_h\in \mathbb{N}$ to be such that $|h|\in [\mfr^{-m_h-1},\mfr^{-m_h}]$.
Then, writing $f= f_{m_h} + \sum_{n\geq m_h} f_{n+1}-f_{n}$
it holds that 
\begin{align*}
\Big\| \frac{f(xh)-\opP^\alpha_x[f](h)}{w(x) |h|^\alpha}\Big\|_{L^p_x} 
&\leq \Big\| \frac{f_{m_h}(xh)-\opP^\alpha_x[f_{m_h}](h)}{w(x) |h|^\alpha}\Big\|_{L^p_x} + \sum_{n\geq {m_h}}
\Big\| \frac{[f_{n+1}-f_{n}](xh)-\opP^\alpha_x[f_{n+1}-f_{n}](h)}{w(x) |h|^\alpha}\Big\|_{L^p_x}\\
&\lesssim\Big\| \frac{f_{m_h}(xh)-\opP^\alpha_x[f_{m}](h)}{w(x) |h|^\alpha}\Big\|_{L^p_x} + \sum_{n\geq m_h} \sum_{d(I)<\alpha}\Big\| \frac{X^{I}[f_{n+1}-f_{n}]}{w(x) |h|^{\alpha-d(I)}}\Big\|_{L^p_x}
\end{align*}
For the first term we find that
\begin{align*}
&\Big\| \frac{f_{{m_h}}(xh)-\opP^\alpha_x[f_{{m_h}}](h)}{w(x) |h|^\alpha}\Big\|_{L^p_x}
\lesssim \sum_{|I|\leq [\alpha]+1, d(I)\geq \alpha} \int_{\mbG}
\Big\| \frac{X^{I}f_{m_h}(xz)}{w(x) }\Big\|_{L^p_x}
\frac{|Q^I|(h, \dd z)}{|h|^\alpha} \\
&\lesssim  \sum_{|I|\leq [\alpha]+1, d(I)\geq \alpha}
\Big\| \frac{X^{I}f_{m_h}}{w }\Big\|_{L^p_x} |h|^{d(I)-\alpha}
\lesssim\sum_{|I|\leq [\alpha]+1, d(I)\geq \alpha}
\Big\| \frac{X^{I}f_{m_h}}{w }\Big\|_{L^p_x} \mfr^{-m_{h}(d(I)-\alpha)}
\end{align*}
and thus 
 $\Big(\int_{|h|<1} \Big\| \frac{f_{m_h}(xh)-\opP^\alpha_x[f_{m_h}](h)}{w(x) |h|^\alpha}\Big\|^q_{L^p_x} \frac{dh}{|h|^{|\fraks|}}\Big)^{1/q}$ is bounded by the right-hand side of \eqref{equiv2}.
 
Next observe that by \eqref{eq:bounding_by_high_derivative} 
\begin{align*}
&\sum_{d(I)<\alpha} \Big(\int_{|h|<1} \Big(\sum_{n\geq m_h} \Big\| \frac{X^{I}[f_{n+1}-f_{n}]}{w(x) |h|^{\alpha-d(I)}}\Big\|_{L^p_x}\Big)^q\frac{dh}{|h|^{|\fraks|}}\Big)^{1/q}\\
&\lesssim
\sum_{\substack{|K|\leq [{\alpha}]+1, \\ d(K)>{\alpha}}}\sum_{d(I)< {\alpha}} \Big(\int_{|h|<1}  \Big(\sum_{n\geq m_h} \Big\|\frac{X^K f_{n+1}}{w} \Big\|_{L^p} \frac{\mfr^{n(d(I)-d(K))} }{|h|^{\alpha-d(I)}}
 \Big)^q \frac{dh}{|h|^{|\fraks|}}\Big)^{1/q}\\
&=\sum_{\substack{|K|\leq [{\alpha}]+1, \\ d(K)>{\alpha}}}\sum_{d(I)< {\alpha}}
\Big(\int_{0}^{1}\Big(
\tilde T_{\alpha-d(I)}\big[ \big\{ a^K_m\big\}_m \big](h)\Big)^q
\frac{dh}{|h|^{|\fraks|}}	\Big)^{1/q}\   
 \end{align*} 
for 
$a^K_n
:=
\Big\|\frac{X^K f_{n+1}}{w}\Big\|_{L^p}
\mfr^{n(\alpha-d(K))}$
and is thus bounded by the right-hand side of \eqref{equiv2} by Lemma~\ref{lem:operatorbound2}.
We thus conclude by Theorem~\ref{prop:equiv} that \eqref{eq:difference} is bounded by the $\|f\|_{B^{\alpha,w}_{pq}}$.

To see the reverse inequality, the argument for the polynomial part is clear. To treat the seminorm, we note that 
$$\langle f-\tilde{\opP}^\alpha_{x}, \phi^\lambda_x\rangle= \int \big(f(y)-\tilde{\opP}^\alpha_{x}[f](y)\big) \phi^\lambda(x^{-1}y) dy
= \int \big(f(xh)-\tilde{\opP}^\alpha_{x}[f](xh)\big) \phi^\lambda(h) dh
$$
and thus
\begin{align*}
\Big\|\sup_{\substack{\phi\in C_{c}(B_1)\\ \| \phi\|_{L^\infty}\leq 1}}\frac{|\langle f-\tilde{\opP}^\alpha_{x}, \phi^\lambda_x\rangle|}{w(x)\lambda^{\alpha}}\Big\|_{L^p}&\leq \int_{\mathbb{G}} \Big\|\frac{|f(xh)-{\opP}^\alpha_{x}[f](h)|}{w(x)\lambda^\alpha}\Big\|_{L^p}\sup_{\substack{\phi\in C_{c}(B_1)\\ \| \phi\|_{L^\infty}\leq 1}}|\phi^\lambda(h)| dh\\
&\leq \int\Big\|\frac{|f(xh)-{\opP}^\alpha_{x}[f](h)|}{w(x)|h|^\alpha}\Big\|_{L^p} \frac{|h|^{\alpha}}{\lambda^{\alpha}}\sup_{\substack{\phi\in C_{c}(B_1)\\ \| \phi\|_{L^\infty}\leq 1}} |\phi^{\lambda}(h)| {dh}\\
&= T'_\alpha [A](\lambda)
\end{align*}
where $T'_\alpha$ is the operator defined in \eqref{operator3} we have set $$A(h):= \Big\|\frac{|f(xh)-{\opP}^\alpha_{x}[f](h)|}{w(x)|h|^\alpha}\Big\|_{L^p} \ , $$ 
and we thus conclude by Lemma~\ref{lem:operatorbound3}

\end{proof}

\subsection{Young Multiplication}

\begin{theorem}\label{prop:young}
Let $p,p_1,p_2, q, q'\in [1,\infty]$ and $\alpha,\beta\in \mathbb{R}\setminus \triangle$ be such that $\beta>\alpha$, $\alpha+\beta>0$ and
$$
\frac{1}{p}= \frac{1}{p_1}+ \frac{1}{p_2} \ .
$$
Then, for any weights $w_1,w_2$ and $w:=w_1 w_2$ the pointwise product on smooth functions extends to a continuous bilinear map
$$
B^{\beta,w_1}_{p_1 q'}\times B^{\alpha,w_2}_{p_2q}\to B^{\alpha,w}_{pq} \ .
$$
\end{theorem}

\begin{proof}
We first consider the easier case $\alpha>0$ using the characterisation of Theorem~\ref{prop:difference}. One readily sees that 
$$\sum_{d(K)<\alpha}
\Big\| \frac{X^K (f_1f_2)}{w}\Big\|_{L^p} \lesssim \sum_{d(I),d(J)<\alpha } \Big\| \frac{X^I f_1}{w_1}\Big\|_{L^{p_1}} \Big\| \frac{X^J f_2}{w_2}\Big\|_{L^{p_2}} \ .
$$
To estimate the seminorm, set $P_x(h):=\opP_x^{\alpha}[f_1](h)\opP_x^{\alpha}[f_2](h)$. Then
\begin{equ}\label{eq:minus_pol}
\big(f_1f_2\big)(xh)-  P_x(h)= 
\big(f_1(xh)-\opP_x^{\alpha}[f_1](h)\big) f_2(xh) +  \opP_x^{\alpha}[f_1](h)\big(
f_2(xh)-  \opP_x^{\alpha}[f_2](h)\big)
\end{equ}
and proceeds using H\"older's inequality. We conclude by noting that the right-hand side of \eqref{eq:minus_pol} satisfies the desired bound, and while $P_x(h)$ does not have homogeneous degree bounded by $\alpha$, its truncation to order $\alpha$ satisfies the same estimate.

We turn to the case $\alpha\leq 0$. Write $f:=f_1$ and $\xi:= f_2$.  
Suppose both involved distributions are actually smooth functions, then
\begin{align*}
(f\xi)_m&= f(\xi)_m + \sum_{n=1}^{N} \Big((f\xi_{m+n} )*\tilde{\rho}^{(m+n-1)} -f \xi_{m+n-1}  \Big) *\tilde{\rho}^{(m+n-2,m)}
\\ &\qquad + \Big((f\xi)_{m+N}- f\xi_{m+N}\Big) * \tilde{\rho}^{(N-1+m,m)} \ .
\end{align*}
Since in this case the last term vanishes as $N\to \infty$, we make the ansatz
\begin{equ}\label{eq:summands}
F_m:=  f(\xi)_m + \sum_{n=1}^{\infty} \Big((f\xi_{m+n} )*\tilde{\rho}^{(m+n-1)} -f \xi_{m+n-1}  \Big) *\tilde{\rho}^{(m+n-2,m)} \ .
\end{equ}
Assuming the series converges in $\mathcal{D}'$, we see that
\begin{equs}
F_{m+1}*\tilde{\rho}^{(m)}&=  \Big(f(\xi)_{m+1} + \sum_{n=1}^{\infty} \Big((f\xi_{m+1+n} )*\tilde{\rho}^{(m+n)} -f \xi_{m+n}  \Big) *\tilde{\rho}^{(m+n-1,m+1)} \Big)*\tilde{\rho}^{(m)}\\
&=  (f(\xi)_{m+1})*\tilde{\rho}^{(m)} + \sum_{n=1}^{\infty} \Big((f\xi_{m+1+n} )*\tilde{\rho}^{(m+n)} -f \xi_{m+n}  \Big) *\tilde{\rho}^{(m+n-1,m)} \\
&=  (f(\xi)_{m+1})*\tilde{\rho}^{(m)} + \sum_{n=1}^{\infty} (f\xi_{m+1+n} )*\tilde{\rho}^{(m+n,m)} -f \xi_{m+n}*\tilde{\rho}^{(m+n-1,m)} \\
&= (f(\xi)_{m})*\tilde{\rho}^{(m)} + \sum_{n=0}^{\infty} (f\xi_{m+1+n} )*\tilde{\rho}^{(m+n,m)} -f \xi_{m+n}*\tilde{\rho}^{(m+n-1,m)} \\
&= F_m \ .
\end{equs}
We next show that the series \eqref{eq:summands} indeed converges and that $\Big\|\mfr^{ m\alpha} \Big \| \frac{F_m}{w}\Big\|_{L^p}\Big\|_{\ell^{q}} \lesssim \| f\|_{B^{\beta,w_1}_{p_1 q'}}\| \xi\|_{B^{\alpha,w_2}_{p_2 q}}$. This will conclude the proof by Theorem~\ref{prop:equiv}.

We assume without loss of generality that $q'=\infty$ in view of Corollary~\ref{prop:embedding}, Item~\ref{prop_item:embedding}.
Observe that we immediately have 
\begin{equ}\label{main_cont}
A_m:=\Big \| \frac{f(\xi)_m}{w}\Big\|_{L^p}  \leq  \Big \| \frac{f}{w_1 }\Big\|_{L^{p_1}}   \Big \| \frac{\xi_m}{ w_2}\Big\|_{L^{p_2}} \ .
\end{equ}
It remains to bound the summands in \eqref{eq:summands}
$$
S_{mn}(z):=
\Big((f\xi_{m+n} )*\tilde{\rho}^{(m+n-1)} -f \xi_{m+n-1}  \Big) *\tilde{\rho}^{(m+n-2,m)}(z) \ .
$$
Since
\begin{align*}
(f\xi_{m+n} )*\tilde{\rho}^{(m+n-1)}(x) -f \xi_{m+n-1}(x)&=
\int f(y)\xi_{m+n}(y) {\rho}^{(m+n-1)}_x(y)dy -f(x) \int \xi_{m+n}(y){\rho}^{(m+n-1)}_x(y) dy\\
&= \int \big[f(y)-f(x) ]\xi_{m+n}(y) {\rho}^{(m+n-1)}_x(y)dy\\
\end{align*}
Thus using the convention ${\rho}_z^{(m,m-1)}:=1$ 
\begin{align*}
S_{mn}(z)&= \int_{\mathbb{G}} \int_{\mathbb{G}} \big[f(y)-f(x) \big] {\rho}^{(m+n-1)}_x(y) {\rho}_z^{(m,m+n-2)}(x)  \xi_{m+n}(y)dx dy \\
&= \int_{\mathbb{G}} \int_{\mathbb{G}} \big[\opP^\beta_y[f](y^{-1}x)- f(x)\big]
 {\rho}^{(m+n-1)}_x(y) {\rho}_z^{(m,m+n-2)}(x)  \xi_{m+n}(y)dx dy\\
&\qquad -\int_{\mathbb{G}} \int_{\mathbb{G}} \big[\opP^\beta_y [f](y^{-1}x) -f(y)\big] {\rho}^{(m+n-1)}_x(y) {\rho}_z^{(m,m+n-2)}(x)  \xi_{m+n}(y)dx dy\\
&=: S^{(1)}_{mn}(z)+S^{(2)}_{mn}(z)
\end{align*}

We note that $S^{(1)}_{mn}(z)= F_{mn}*\tilde{\rho}^{(m+n-2,m)}(z)$ for 
\begin{align*}
F_{mn}&= 
\int \big[\opP^\beta_y[f](y^{-1}x)- f(x)\big]\xi_{m+n}(y) {\rho}^{(m+n-1)}_x(y)dy\\
&= 
\int \big[\opP^\beta_{xh^{-1}}[f](h)- f(x)\big]\xi_{m+n}(xh^{-1}) {\rho}^{(m+n-1)}(h^{-1})dh
\end{align*}
Thus $\Big\|\frac{S^{(1)}_{mn}}{w}\Big\|_{L^p}\lesssim \Big\|\frac{F_{mn}}{w}\Big\|_{L^p}$ and
\begin{equs}
\Big\|\frac{F_{mn}}{w}\Big\|_{L^p}
&\leq \int_{\mathbb{G}} \Big\|\frac{\big[\opP^\beta_{xh^{-1}}[f](h)- f(x)\big]\xi_{m+n}(xh^{-1})}{w(x)}\Big\|_{L^{p}_x} \, | {\rho}^{(m+n-1)}(h^{-1}) | \, dh\\
&\lesssim
\int_{\mathbb{G}} \Big\|\frac{\big[\opP^\beta_{x}[f](h)- f(xh)\big]\xi_{m+n}(x)}{w(x)}\Big\|_{L^{p}_x} \, | {\rho}^{(m+n-1)}(h^{-1}) | \, dh\\
&\leq \int_{\mathbb{G}} \Big\|\frac{f(xh)-\opP^\beta_x [f](h)}{w_1(x)} \Big\|_{L_x^{p_1}}  \Big\| \frac{\xi_{m+n}(xh)}{w_2(x)}\Big\|_{L^{p_2}_x} \, | \tilde{\rho}^{(m+n-1)}(h)  |\, dh\\
&\leq\Big\| \frac{\xi_{m+n}}{w_2}\Big\|_{L^{p_2}_x}    \int_{\mathbb{G}} \Big\|\frac{f(xh)-\opP^\beta_x [f](h)}{w_1} \Big\|_{L_x^{p_1}} \, | \tilde{\rho}^{(m+n-1)}(h) |\, dh \ .
\end{equs}
Since
\begin{align*}
 &\int_{\mathbb{G}} \Big\|\frac{f(xh)-\opP^\beta_x [f](h)}{w_1} \Big\|_{L_x^{p_1}}\, |  \tilde{\rho}^{(m+n-1)}(h)| \, dh 
 \lesssim \mfr^{-(m+n)\beta} \|f \|_{B^{\beta,w_1}_{p_1 q'}}
\end{align*}
it follows that 
\begin{equ}\label{remainder}
B_{m}:=\sum_{n} \Big\|\frac{F_{mn}}{w}\Big\|_{L^p} 
\leq  \|f\|_{B^{\beta,w_1}_{p_1 q'}} \sum_{n} \mfr^{-(m+n)\beta} \Big\| \frac{\xi_{m+n}}{w_2}\Big\|_{L^{p_2}_x}
\lesssim \mfr^{-(\beta+\alpha)m} \| \xi\|_{B^{\alpha,w_2}_{p_2q}} \|f\|_{B^{\beta,w_1}_{p_1 q'}} \ .
\end{equ}
%
%
%
%
For the term $S^{(2)}_{mn}$
we note that $S^{(2)}_{m1}(z)=0$, while for $n\geq 2$
\begin{align*}
&S^{(2)}_{mn}(z)\\
&= \int_{\mathbb{G}} \int_{\mathbb{G}} \big[\opP^\beta_y [f](h) -f(y)\big] {\rho}^{(m+n-1)}(h^{-1}) {\rho}_z^{(m,m+n-2)}(yh)  \xi_{m+n}(y)dh dy\\
 &= \sum_{0<d(I)<\beta} 
\int_{\mathbb{G}} \int_{\mathbb{G}} c_{I}(y)\eta^{I}(h)  {\rho}^{(m+n-1)}(h^{-1}) {\rho}_z^{(m,m+n-2)}(yh)  \xi_{m+n}(y)dh dy\\
 &= \sum_{0<d(I)<\beta} 
\int_{\mathbb{G}} \int_{\mathbb{G}} c_{I}(y)\eta^{I}(h)  {\rho}^{(m+n-1)}(h^{-1}) \Big[{\rho}_z^{(m,m+n-2)}(yh)-\opP^{\beta-d(I)}_y[{\rho}_z^{(m,m+n-2)}](h) \Big]   \xi_{m+n}(y)dh dy
\end{align*}
Noting that 
$$
|\eta^{I}(h)|\Big|{\rho}_z^{(m,m+n-2)}(yh)-\opP^{\beta-d(I)}_y[{\rho}_z^{(m,m+n-2)}](h) \Big|\lesssim \mathbf{1}_{|z^{-1}y|<\mfr^{-m}} \sum_{M>\beta}|h|^{M}\mfr^{m(M+|\fraks|)}
$$
where the sum over $M$ runs over some finite subset of $\mathbb{R}$, and setting 
$$
G^{I}_{mn}(y):=
\int_{\mathbb{G}} \Big|c_{I}(y) {\rho}^{(m+n-1)}(h^{-1})  \Big[\sum_{M>\beta}|h|^{M}\mfr^{mM}\Big]  \xi_{m+n}(y)\Big|
dh
\lesssim \mfr^{-\beta n} |c_{I}(y) \xi_{m+n}(y) | 
$$
it follows that $$\Big\|\frac{S^{(2)}_{mn}}{w}\Big\|_{L^p}\lesssim \sum_{0<d(I)<\beta} \Big\|\frac{G^I_{mn}}{w}\Big\|_{L^p}\lesssim \mfr^{-\beta n}  \sum_{{0<d(I)<\beta}}\Big\|\frac{c_I}{w_1}\Big\|_{L^{p_1}}\Big\|\frac{ \xi_{m+n}}{w_2}\Big\|_{L^{p_2}}
\lesssim 
\mfr^{-\beta n}  \| f\|_{B^{\beta,w_1}_{p_1 q'}}  \Big\|\frac{ \xi_{m+n}}{w_2}\Big\|_{L^{p_2}} \ ,
 $$
where we have used that the Taylor coefficients are linear combinations of derivatives. Thus,
\begin{equ}\label{eq:s2etimate}
 C_m:=\sum_{n}\Big\|\frac{S^{(2)}_{mn}}{w}\Big\|_{L^p} \lesssim 
 \mfr^{-m \alpha }\| f\|_{B^{\beta,w_1}_{p_1 q'}} \sum_{n} \mfr^{-(\alpha+\beta) n}  \Big[\mfr^{\alpha(n+m)} \Big\|\frac{ \xi_{m+n}}{w_2}\Big\|_{L^{p_2}}\Big] \ .
\end{equ}
Thus it remains to bound the $\ell^{q}$-norm
of 
$
\big(\mfr^{\alpha m}A_m\big)_{m\in \mathbb{N}}
$,
$\big(\mfr^{\alpha m}B_m\big)_{m\in \mathbb{N}}$
and
$
\big(\mfr^{\alpha m}C_m\big)_{m\in \mathbb{N}}$  from 
 \eqref{main_cont},
 \eqref{remainder} and 
\eqref{eq:s2etimate} respectively.
The desired estimate on the first and second term follow directly, while for the last term it follows using the discrete Young inequality.
\end{proof}

\subsection{Smoothing properties of singular kernels}

\begin{assumption}\label{ass:kernel}
For $T>0$, assume that we are given kernels $K_t: \mathbb{G}\setminus\{e\}\to \mathbb{R}$ for $t>0$ satisfying for some $\kappa>0$ and $m>1$
\begin{equ}\label{eq:ass_upper}
|\partial^j_t X^I K_t (x)|\lesssim_{I,j,T} {t^{-\frac{|\fraks|+d(I)}{m}-j }}\exp\big(-\kappa (|x|^m/t)^{1/(m-1)}\big) 
\end{equ}
uniformly over $x\in \mathbb{G}$ and $t\in (0,T]$.
Furthermore, assume that $\int K_t =1$ as well as 
\begin{equ}\label{eq:ass_poly}
  \sup_{0<t\le T}
  \left|
  \partial_t \int_{\mathbb G} P(y)K_t(y)\,dy
  \right|
  <\infty \ 
\end{equ}
for all $P\in \mathcal{P}_m$.
\end{assumption}

\begin{remark}
Note that by \eqref{eq:re_expand_basis} this bound immediately implies the same bound for $K$ replaced by $\tilde{K}$ at the expense of slightly reducing the value of  $\kappa$.
\end{remark}
\begin{remark}\label{rem:heat_semi}
It is well known that for instance heat semigroups associated to positive Rockland operators satisfy this assumption with $m$ the degree of the differential operator, we refer to \cite[Thm.~10]{Nice_noteDHZ94} and \cite[Prop.~5.5]{Elst_Robinson} for the bound \eqref{eq:ass_upper}. 
The bound \eqref{eq:ass_poly} follows using the associated PDE.
\end{remark}
For a kernel satisfying Assumption \ref{ass:kernel} and any sufficiently integrable function $f$ we set
\begin{equ}\label{eq:kernel_operator}
 P_tf(x)= (f*K_t)(x)= \int K_t(y^{-1}x) f(y)dy\ , 
 \end{equ}
and extend it to distributions in the usual way.
In the sequel we shall make the following assumption on the weight $w$: for every $\kappa'>0$
\begin{equ}\label{assump:weight}
 \sup_{y\in \mathbb{G}} \int \exp\big(-\kappa' |y^{-1}x|^{m/(m-1)}\big) \Big(\frac{ w(y) }{w(x)} +\frac{ w(x) }{w(y)}\Big) dx  <\infty \ .
\end{equ}


\begin{prop}\label{prop:schauder}
Let $T>0$, $w$ be a weight satisfying \eqref{assump:weight} and $p,q\in [1,\infty]$.
Then, for every $\alpha,\beta\notin \triangle$, $\alpha<\beta$ and $\gamma\in (0,m)$ such that $\beta+\gamma\notin\triangle$ 
$$\|P_t f\|_{B^{\beta, w}_{p,q}}\lesssim t^{-\frac{\beta-\alpha}{m}} \|f\|_{B^{\alpha, w}_{p,q}} \ , \qquad \|P_t f-f\|_{B^{\beta, w}_{p,q}}\lesssim t^{\frac{\gamma}{m}} \|f\|_{B^{\beta+\gamma, w}_{p,q}},$$
uniformly over $f\in B^{\alpha, w}_{p,q}$ and
 $t \in (0,T]$.
\end{prop}

\begin{remark}
Note that \eqref{assump:weight} holds locally uniformly in $a, \ell\in \mathbb{R}$ for the weights $p_a$ and $e_{\ell}$ in \eqref{eq:weights_examle}. In this case the implicit constants in Proposition~\ref{prop:schauder} can also be chosen locally uniformly in $a, \ell\in \mathbb{R}$.
%
%

\end{remark}


In order to separate regularity and integrability considerations write
$$ K_t(x)= K^{(+)}_t(x) + K^{(-)}_t(x) \ ,$$
where for fixed $\phi\in C^\infty_c \big((-1,1)\times B_{1})$ such that $\phi=1$ for $(t,x)\in (-1/2,1/2)\times B_{1/2}$,
$$K^{(+)}_t(x):= \phi(t,x)K_t(x), \qquad K^{(-)}_t(x):= (1-\phi(t,x))K_t(x)\ .$$
We also write $P^{(+)}_t$, resp.\ $P^{(-)}_t$ for the operator defined analogously to \eqref{eq:kernel_operator}. 
Then, the proof is a direct consequence of Lemma~\ref{lem:largeSchauder} and Lemma~\ref{lem:smallSchauder} below.

\begin{lemma}\label{lem:largeSchauder}
In the setting of Proposition~\ref{prop:schauder} it holds that
$$\|P^{(-)}_t f\|_{B^{\beta, w}_{p,q}}\lesssim \|f\|_{B^{\alpha, w}_{p,q}}\ , \qquad\|\partial_t P^{(-)}_t f\|_{B^{\beta, w}_{p,q}}\lesssim \|f\|_{B^{\alpha, w}_{p,q}} $$
uniformly over $t \in (0,T]$.
\end{lemma}
\begin{proof}
First consider the case when $f$ is function valued, i.e. $\alpha>0$ and note that by Young's inequality
$$\Big\| \frac{f* X^{I} K^{(-)}_t}{w} \Big\|_{L^p}\lesssim \Big\|\frac{f}{w}\Big\|_{L^p} \ ,$$
if
\begin{equ}
 \sup_{y\in \mathbb{G},t\in (0,T]} \int \frac{|X^{I} K^{(-)}_t(y^{-1}x)| w(y) }{w(x)} dx <\infty,
\qquad 
\sup_{x\in \mathbb{G},t\in (0,T]} \int \frac{|X^{I} K^{(-)}_t(y^{-1}x)| w(y) }{w(x)} dy <\infty \ 
\end{equ}
which in turn is a straightforward consequence of the condition \eqref{assump:weight}.

We turn to the case when $\alpha<0$ we write for $f_n$ as in Theorem~\ref{prop:equiv} $f=f_0+ \sum_{n} f_{n+1}-f_n$.
The argument above provides the desired estimate on $f_0$. 
For the increment
\begin{equ}\label{eq:Young_to_use}
P^{(-)}_t (f_{n+1}-f_{n})= f* (\tilde{\varphi}^{(n+1)}-\tilde{\varphi}^{(n)})* K^{(-)}_t
=f_{n+1}*( K^{(-)}_t-\tilde{\rho}^{(n)}*K^{{(-)}}_t) 
\end{equ}
note that for any $M>0$ 
%
\begin{align*}
K^{(-)}_t(x)-\tilde{\rho}^{(n)}*K^{{(-)}}_t(x)
&= \int \tilde{\rho}^{(n)}(h^{-1})\big( K^{(-)}_t(x)- K^{(-)}_t(h^{-1}x)\big) dh \\
&= \int {\rho}^{(n)}(h)\big( \tilde{K}^{(-)}_t(x^{-1})- \opP^{M}_{x^{-1}}[\tilde{K}^{(-)}_t](h)\big) dh
\\
&=\sum_{{|I|\leq [M]+1, d(I)\geq M}} \int {\rho}^{(n)}(h)\big( X^{I}\tilde{K}^{(-)}_t(x^{-1}z)Q^I(h, dz)\big) dh
\end{align*}
and thus for any $M>0$,  
$$\int\frac{|X^{J}K^{(-)}_t(y^{-1}x)-X^{J}(\tilde{\rho}^{(n)}*K^{{(-)}}_t)(y^{-1}x)| w(y)}{w(x)}dx
\lesssim 
\mfr^{-Mn}
$$
Thus using Young's inequality in \eqref{eq:Young_to_use} it follows that 
$$\Big\|\frac{X^{J} P^{(-)}_t (f_n-f_{n+1})}{w} \Big\|_{L^{p}}\lesssim \mfr^{-nM} \Big\|\frac{f_{n+1}}{w}\Big\|_{L^{p}}
$$
Finally, the very same argument applies to $\partial_t P^{(-)}_t$.
\end{proof}
\begin{lemma}\label{lem:smallSchauder}
In the setting of Proposition~\ref{prop:schauder} it holds that
\begin{equ}\label{eq:small_scaleSchauder}
\|P^{(+)}_t f\|_{B^{\beta, w}_{p,q}}\lesssim t^{-\frac{\beta-\alpha}{m}} \|f\|_{B^{\alpha, w}_{p,q}} \ , 
\qquad \|P_t^{(+)} f-f\|_{B^{\beta, w}_{p,q}}\lesssim t^{\frac{\gamma}{m}} \|f\|_{B^{\beta+\gamma, w}_{p,q}},
\end{equ}
uniformly over $t \in (0,T]$.
\end{lemma}

\begin{proof}
Fix a smooth dyadic partition of unity $\mathbf{1}_n$
adapted to the annuli
$$
  A_n:= \Big\{     (t,x)\in \mathbb{R}\times \mathbb{G}\ :\   |t|^{1/m}+|x|\in (\mfr^{-n-1}, \mfr^{-n} ) \Big\}
$$
and set $K^{(n)}_t(x):= \mathbf{1}_n(t,x) \cdot K^{(+)}_t(x)$. 
Set $N_t\in \mathbb{N}$ for $t\in (0,1]$ to be such that $t^{1/m}\in (\mfr^{-N_t-1},\mfr^{-N_t}]$.
Then, by \eqref{eq:ass_upper}
 for every $I,k$ and $A\geq 0$,
\begin{equ}\label{eq:component_upper_bound}
 |\partial^k_t X^{I} K^{(n)}_t|\lesssim_{A,I,k} \mfr^{-A(N_t-n)} \mfr^{n(|\fraks|+d(I)+km)}
\end{equ} 
  uniformly in $n$ and $t$ 
 and  $K^{(+)}_t(x)=\sum_{n=0}^{N_t} K^{(n)} (t,x)\ .$ 
Set $\Psi_t^{k,n}:=K^{(n)}_t\ast \tilde{\varphi}^{(k)}$
and $\Xi_t^{(k)}= \sum_{n=k+1}^{N_t} \Psi_t^{k,n}$.
  Note that 
   $$(f*K^{(+)}_t)_k =f*K^{(+)}_t\ast \tilde{\varphi}^{(k)}= \sum_{n=0}^{N_t\wedge k} f*\Psi_t^{k,n} +  f*\Xi^{(k)}_t \ .
   $$ 
 Note that $\Psi_t^{k,n}$ is supported at scales of order $\mfr^{-(k\wedge n)}$. The function $\Xi_t^{(k)}$ is supported at scales of order $\mfr^{-k}$ and vanishes whenever 
 $k\geq N_t$. It follows from \eqref{eq:component_upper_bound} that for any $A\geq 0$
\begin{equ}\label{eq:rescaled_function_bound}
  | X^{I}\Psi_t^{k,n}|\lesssim_{I,A } \mfr^{-A(N_t-n)}  \mfr^{(k\wedge n)(|\fraks| +d(I)) } \ , \qquad |X^{I}\Xi_t^{(k)}|\lesssim   \mfr^{k(|\fraks| +d(I)) } \ .
\end{equ}
uniformly in $n,k,I$ and $t$. 

 We shall  
 estimate Besov norms using Theorem~\ref{prop:equiv}.
Note that whenever either $\alpha\leq 0$ or $0<\alpha<d(I)$
\begin{equs}   
   \Big\| \frac{X^{I}(f*K^{(+)}_t)_k}{w}\Big\|_{L^p}&\leq \sum_{n=0}^{N_t\wedge k}\Big\| \frac{f*X^{I}\Psi_t^{k,n}}{w}\Big\|_{L^p} + \Big\| \frac{f*X^{I}\Xi^{(k)}_t}{w}\Big\|_{L^p} \\
   &
   \lesssim 
\|f\|_{B^{\alpha,w}_{p,q}}
\Big(\sum_{n=0}^{N_t\wedge k} \mfr^{-(\alpha-d(I))(k\wedge n)}  + \mathbf{1}_{k<N_t} \mfr^{-(\alpha-d(I)) k}\Big)\\
&\lesssim \|f\|_{B^{\alpha,w}_{p,q}} (\mfr^{-(\alpha-d(I)) (k\wedge N_t)} + \mathbf{1}_{k<N_t} \mfr^{-(\alpha-d(I)) k} ), \label{locbound}
\end{equs}   
where in the case $\alpha>0$ we use that for every $P\in \mathcal{P}_\alpha$
$$
P*X^{I}\Psi_t^{k,n}=P*X^{I}\Xi^{(k)}_t=0 \ .
$$

We first prove the first inequality of \eqref{eq:small_scaleSchauder} in the case $\alpha<\beta<0$,
\begin{align*}
\Bigg(
\sum_{k\geq0}
\Big[
\mfr^{k\beta}  \Big\| \frac{(f*K^{(+)}_t)_k}{w}\Big\|_{L^p}\Big]^q
\Bigg)^{1/q}
\lesssim
&\Bigg(
\sum_{k\geq0}
\Big[
\mfr^{k\beta}
(\mfr^{-\alpha (k\wedge N_t)} + \mathbf{1}_{k<N_t} \mfr^{-\alpha k} )
\Big]^q
\Bigg)^{1/q}\\
&\leq 
\Bigg(
\sum_{k\geq0}
\Big[
\mfr^{k\beta}
 \mfr^{-\alpha (k\wedge N_t)} 
\Big]^q
\Bigg)^{1/q}
+
\Bigg(
\sum_{k\geq0}
\Big[
\mfr^{k\beta}
\mathbf{1}_{k<N_t} \mfr^{-\alpha k} 
\Big]^q
\Bigg)^{1/q}\\
&\lesssim
\mfr^{N_t(\beta-\alpha)} \lesssim t^{-\frac{\beta-\alpha}{m}}
\end{align*}
In the case $\beta>0$, again by Theorem~\ref{prop:equiv},
\begin{equs}
\|P^{(+)}_t f\|_{B^{\beta,w}_{p,q}}
&\lesssim
\Big\|\frac{(P^{(+)}_t f)\ast \tilde\varphi}{w}\Big\|_{L^p}
+
\max_{\substack{|L|\leq[\beta]+1\\ d(L)>\beta}}
\Bigg(
\sum_{k\geq0}
\mfr^{kq(\beta-d(L))}
\Big\|
        \frac{X^L\big((P^{(+)}_t f)\ast \tilde\varphi^{(k)}\big)}{w}
\Big\|_{L^p}^q
\Bigg)^{1/q}\\
&\lesssim
 \Big\|\frac{ f\ast\Psi_t^{0,0}}{w}\Big\|_{L^p} + 
 \Big\| \frac{f*\Xi^{(0)}_t}{w}\Big\|_{L^p} \\
&\qquad +
\max_{\substack{|L|\leq[\beta]+1\\ d(L)>\beta}}
\Bigg(
\sum_{k\geq0}
\mfr^{kq(\beta-d(L))}
\Big[\sum_{n=0}^{N_t} \Big\|
        \frac{X^{L} (f \ast K^{(+)})_k}{w}\Big\|_{L^p}\Big]
^q
\Bigg)^{1/q}\ .
\end{equs}
The terms on the first line are easily bounded, while using \eqref{locbound} for each $L$
 reduces to the estimate
$$
\Bigg(
\sum_{k\geq0}
\Big[
\mfr^{k(\beta-d(L))}
(\mfr^{-(\alpha-d(L)) (k\wedge N_t)} + \mathbf{1}_{k<N_t} \mfr^{-(\alpha-d(L)) k} )
\Big]^q
\Bigg)^{1/q} \lesssim \mfr^{N_t(\beta-\alpha)} \ .
$$


We turn to the second inequality of \eqref{eq:small_scaleSchauder}.
Set $\Gamma_{s}^{k,n}:=\partial_s  K^{(n)}_s *\tilde{\varphi}^{(k)}$ and 
$\Theta_s^{(k)}= \sum_{n\geq k+1}^{N_s} \Gamma_{s}^{k,n}$. 
We claim that these functions satisfy
\begin{equ}\label{bounds}
| X^{K}\Gamma_s^{k,n}|\lesssim_{K,A } \mfr^{-A(N_s-n)}  \mfr^{(k\wedge n)(|\fraks| +d(K)) } \mfr^{mn} \ , 
\qquad 
| X^{K}\Theta_s^{(k)}|\lesssim_{K,A }  \mfr^{k (|\fraks| + d(K) +m)}
\end{equ}
Indeed, the first bound follows straightforwardly from \eqref{eq:component_upper_bound}. For the second bound we write 
$
\Theta_s^{(k)}= \partial_s K^{(+)}_s * \phi^{(k)}- \sum_{n\leq k} \Gamma_s^{k,n}\ .
$
The estimate on the sum follows directly from the first bound of \eqref{bounds} and we turn to the former term for which we note
\begin{equs}\label{toboundssss}
X^{K}(\partial_s K^{(+)} * \tilde{\varphi}^{(k)})(x) &= \int_{\mathbb{G}} \partial_s K^{(+)}(y) \big(Y^{K}{\varphi}^{(k)}\big)(x^{-1}y) dy\\
&=\int_{|x^{-1}y|\leq \mfr^{-k}} \partial_s K^{(+)}(y) \big[\big(Y^{K}{\varphi}^{(k)}\big)(x^{-1}y)-\opP^{m}_{x^{-1}}[Y^{K}{\varphi}^{(k)}](y)
\big] dy \\
&\qquad +
\int_{|x^{-1}y|\leq \mfr^{-k}}  \partial_s K^{(+)}(y) \opP^{m}_{x^{-1}}[Y^{K}{\varphi}^{(k)}](y) dy.
\end{equs}
The second summand of \eqref{toboundssss} is clearly bounded by a multiple of $\mfr^{k(m+|\fraks|+d(K))} $
since $$
\Big| \int \partial_s K^{(+)}(y)\eta^{I}(y)dy\Big|\leq \Big| \int \partial_s K(y)\eta^{I}(y)dy\Big| + \Big| \int \partial_s K^{(-)}(y)\eta^{I}(y)dy\Big|<\infty\ 
$$ whenever $d(I)<m$.
For the first summand of \eqref{toboundssss} note that
\begin{align*}
&\sum_{|I|\leq [m]+1, d(I)> m}  \mfr^{k^{(|\fraks|+ d(I) +d(K))}} \int_{|x^{-1}y|\leq \mfr^{-k}} \big| \partial_s K^{(+)}(y)\big|\,  |y|^{d(I)}  dy
\lesssim 
\sum_{|I|\leq [m]+1, d(I)> m} \mfr^{k(m+|\fraks|+d(K))} \mfr^{{k}(m-d(I))} \ .
\end{align*}

We can write similarly to above
$$ (P^{(+)}_t f -P^{(+)}_u f) *\tilde{\varphi}^{(k)}=  \sum_{n} \int_{u}^t f* \partial_s  K^{(n)}_s *\tilde{\varphi}^{(k)} ds =  \int_{u}^t  \sum_{n=0}^{N_s\wedge k} f* \Gamma_{s}^{k,n} +  f*\Theta_s^{(k)} ds \ $$
and thus uniformly in $s\in (0,t)$
$$ \big|X^{I}\big(P^{(+)}_t f*\tilde{\varphi}^{(k)}-  P^{(+)}_sf*\tilde{\varphi}^{(k)}\big)\big|\leq    \int_{0}^t \Big| \sum_{n=0}^{N_s\wedge k} X^{I}f*  \Gamma_{s}^{k,n}\Big| ds  +
 \int_{0}^t \Big| X^{I} f*\Theta_s^{(k)}\Big| ds
 \ $$
 and therefore by dominated convergence
\begin{equs}
  & \Big\| \frac{|X^{I}P^{(+)}_t f*\tilde{\varphi}^{(k)}-  X^{I}f*\tilde{\varphi}^{(k)}|}{w}\Big\|_{L^p}
  \leq \int_{0}^t  \Big[\sum_{n=0}^{N_s\wedge k}\Big\| \frac{X^{I} f*\Gamma_{s}^{k,n}}{w}\Big\|_{L^p}
+ \Big\| \frac{X^{I}f*\Theta_{s}^{(k)}}{w}\Big\|_{L^p}\Big] ds  .
\end{equs}   
From here on we can argue as for the first inequality of \eqref{eq:small_scaleSchauder} but using \eqref{bounds} instead of \eqref{eq:rescaled_function_bound}
to conclude that for $\beta<0$
$$
\Bigg(
\sum_{k\geq0}
\Big[
\mfr^{k\beta}  \Big\| \frac{(P^{(+)}_t f-  f}{w}\Big\|_{L^p}\Big]^q
\Bigg)^{1/q}
\lesssim \int_0^t\mfr^{N_s(m-\gamma)}ds \|f\|_{B^{\beta+\gamma,w}_{p,q}} \lesssim
 t^{\gamma/m} \|f\|_{B^{\beta+\gamma,w}_{p,q}}\ .
$$
The argument for $\beta>0$ is analogous.

\end{proof}

\subsection{A Kolmogorov Criterion}
The Kolmogorov criterion is an important probabilistic result for determining whether a random field has a modification in a specified function or distribution space. The next lemma provides such a criterion adapted to the intrinsic Besov spaces considered in this article. With the SPDE applications of the next section in mind, we
formulate it at the level of an integrated process $X_t$ rather than the space-time noise $\partial_t X= \xi$, as this is more convenient when using mild-sewing arguments.
\begin{lemma}\label{lem:kolmogorov}
Assume that for $C>0, r\in [1,\infty), \alpha\leq 0$, $H\in (0,1]$ a random field\footnote{
We in particular assume we are given a probability space $\Omega$, such that $(\omega, t,\phi)\mapsto X_t({\phi})$ is jointly measurable.}
 $X_t({\phi})=X^{\omega}_t({\phi})$ indexed by $t\in \mathbb{R}$, $\phi\in L^2(\mathbb{G})$ is (space) left-translation invariant in law and satisfies for $\varphi^{(n)}$ as in Assumption~\ref{ass:pair}
$$ \E\big[\big|X_t(\varphi^{(n)})\big|^r\big]\leq C \mfr^{-n\alpha r} \ , \qquad \E\big[\big|X_t(\varphi^{(n)})-X_s(\varphi^{(n)})\big|^r\big]\leq C|t-s|^{Hr} \mfr^{-n\alpha r}$$
uniformly over $s,t\in [0,T]$.
Then, if $1/r<H$, there exists a modification for which the map $t\mapsto X_t\in \mathcal{D}'$ is continuous and satisfies for every $\bar{\alpha}<\alpha-|\fraks|/r$,
and $\bar{H}<H-1/r$,
 every $p\in [1,r]$ and $q\in [1,\infty]$, and weight $w$ satisfying $w(x)^{-1}\in {L}^p(\mathbb{G})$
 $$ 
\sup_{t\in [0,T]}\|X_{t}\|_{B^{\bar{\alpha},w}_{pq}}+ \sup_{\substack{t,s\in [0,T]\\t\neq s}}\frac{ \|X_t-X_s\|_{B^{\bar{\alpha},w}_{pq}}}{|t-s|^{\bar{H}}} \lesssim_\omega \|w^{-1}\|_{L^p}, \qquad a.s.
$$ 
\end{lemma}
\begin{proof}
For $n\geq 0$ set
$
X_t^{(n)}(x):=(X_t,\varphi^{(n)}_x) \ .
$
By spatial stationarity,
\[
 \E |X_t^{(n)}(x)|^r \lesssim \mfr^{-n\alpha r},
\qquad
 \E |X_t^{(n)}(x)-X_s^{(n)}(x)|^r \lesssim |t-s|^{Hr}\mfr^{-n\alpha r},
\]
uniformly in $x\in\mathbb G$. 
Since $r/p\ge 1$,
\begin{equs}
\E \big[\|X_t^{(n)}\|_{L^p_w}^r\Big] &=
\E \Big(\int_{\mathbb G} |X_t^{(n)}(x)|^p w(x)^{-p}\,dx\Big)^{r/p}
\leq
\Big(\int_{\mathbb G} (\E |X_t^{(n)}(x)|^r)^{p/r} w(x)^{-p}\,dx\Big)^{r/p}.
\\
&\leq C \mfr^{-n\alpha r}\|w^{-1}\|_{L^p}^r.
\end{equs}
and similarly 
$
\E \|X_t^{(n)}-X_s^{(n)}\|_{L^p_w}^r
\leq C
|t-s|^{Hr}\mfr^{-n\alpha r}\|w^{-1}\|_{L^p}^r.
$

Thus it follows by Theorem~\ref{prop:equiv} that almost surely 
$X_t= \lim X^{(n)}_t$ exists and that
for every $\bar\alpha<\alpha-|\mathfrak s|/r$,
\[
\E \|X_t\|_{B^{\bar\alpha,w}_{p,q}}^r \lesssim C \|w^{-1}\|_{L^p}^r,
\qquad
\E \|X_t-X_s\|_{B^{\bar\alpha,w}_{p,q}}^r
\leq C |t-s|^{Hr}\|w^{-1}\|_{L^p}^r.
\]
The proof is complete by applying standard Banach space valued Kolmogorov.
\end{proof}

\begin{remark}
Important examples include fractional in time, white in space noises $W^{H}$ for $H\in (0,1)$. This is characterised as the centered Gaussian field with covariance 
$$
\E\big[W^{H}_t({\phi})W^{H}_s({\psi})\big] =\frac{|t|^{2H}+|s|^{2H}- |t-s|^{2H}}{2} \langle \phi, \psi\rangle_{L^{2}(\mathbb{G})}   \ .
$$
Fractional in space variants include noises of the form 
$K*W^{H}_t$ obtained as convolution in space with a singular kernel $K$, see \cite[Sec.~6]{MS25} as well as 
 $(\mathcal{R} +c)^{\theta} W^{H}_t$ for $\theta\in \mathbb{R}$, $c> 0$ or $\mathcal{R}^{\theta} W^{H}_t$ for $\theta>-\frac{|\fraks|}{2m}$.
These are easily checked to satisfy the assumptions of Lemma~\ref{lem:kolmogorov} for any $r\geq 2$ and $\alpha<(-m\theta- |\fraks|/2)\wedge 0$.

The last example includes the noises considered in \cite{Tindel_Big, Tindel_ito} in the special case when $\mathbb{G}$ is the Heisenberg group and $\mathcal{R}$ the sub-Laplacian.
\end{remark}

\section{Applications to Parabolic Anderson Models}\label{sec_applications to SPDE}
In this section, we illustrate the above considerations on function spaces by establishing the well-posedness of SPDEs 
of the form \eqref{eq:illustrate}
on $\mathbb{R}\times \mathbb{G}$. 
Standard power counting from Euclidean space naturally adapts to this setting:
Assuming that the operator $\mathcal{R}$ has degree $m$, it is natural to interpret the direct product $\mathbb{R}\times \mathbb{G}$ as a homogeneous Lie group  with the dilation $\bar{\mathfrak{s}}$, such that $\bar{\mathfrak{s}} \partial_t= m \partial_t$ and $\bar{\mathfrak{s}}X_i= {\mathfrak{s}}X_i$.
Then a necessary condition for this equation to belong to the Young regime, i.e.\ all products on the right-hand side can be made sense of by Young multiplication, is given
for $\alpha^{I}$ the space-time regularity of $\xi^{(I)}$ by
\begin{equ}\label{eq:young_regime}
\min_{I\in \mathcal{K}} \alpha^I + \min_{K\in \mathcal{K}} (\alpha^{K}-d(K)) +m >0 \ ,
\end{equ}
for $\mathcal{K}:=\big\{K \ : \ d(K)<m, c_K\neq 0\big\}$.
%

%
%

%

\subsection{A Sewing Lemma}

We present a mild sewing lemma in the spirit of \cite{Gub_tindel_sewing}.
Technical novelties, besides the formulation being closer to \cite{friz_hairer_20_introduction}, are that we work with norms $|\cdot |_{\alpha,t}$ indexed by $\alpha\in \mathbb{R}$ and $t\in [0,T]$ (where the index $t$ will allow to work with time-dependent weights), and that we shall formulate a variant which allows norms to blow up close to $t=0$ (to allow for rougher initial conditions when applied to SPDEs).

We fix Banach spaces $(V_{\alpha,t},|\cdot |_{\alpha,t})$ such that $V_{\beta,s}\subset V_{\alpha,t}$ whenever $s\leq t, \beta\geq \alpha$ with embedding norm at most $1$. We shall assume we are given a semigroup $S_t$ acting on these spaces, satisfying 
\begin{equ}\label{eq:semi-group_bounds}
|S_t u|_{\alpha+\beta, T} \lesssim t^{-\beta} |u|_{\alpha, T}, \qquad |S_t u- u|_{\alpha-\gamma, T} \lesssim t^{\gamma} |u|_{\alpha, T} \ ,
\end{equ}
uniformly in $\gamma\in [0,1]$ and $u\in V_{\alpha,T}$, as well as uniformly over $t,T$, $\beta\geq 0 , \alpha$ in compact sets.

We shall work with increments $\Xi_{t,s}\in V_{\alpha, t}$ for $t\geq s>0$, and write 
$$(\hat{\delta} \Xi)_{t,u,s} 
= \Xi_{t,s} -\Xi_{t,u}- S_{t-u}\Xi_{u,s}  \ .
$$

\begin{definition}
For $\sigma,\mu>0$, a finite set $J$,  $\eta=(\eta_1, (\eta_{2,j})_{j\in J})$ and $\rho=(\rho_j)_{j\in J}$ where $\eta_1, \eta_{2,j}, \rho_j \geq 0$ and $\alpha\in \mathbb{R}$, 
we define $C^{\sigma,\mu,\eta}_{2,\alpha,\rho}((0,T],V)$ 
to be the space of increments
 $\{\Xi_{t,s}\}_{0<s<t\leq T}$  satisfying 
$$\Xi_{t,s}=S_{t-s} \tilde{\Xi}_{t,s}, \qquad  (\hat{\delta} \Xi)_{t,u,s} = \sum_{i} S_{t-u} \tilde{\Xi}^i_{t,u,s}$$
for some $\tilde{\Xi}_{t,s}\in V_{\alpha,t}$ and $\tilde{\Xi}^i_{t,u,s}\in V_{\alpha-\rho_i,t}$ which satisfy 
 
$$
\llbracket  \tilde{\Xi}\rrbracket_{\sigma;\eta_1;T}:= \sup_{\substack{0<s<t\leq T}} \frac{|\tilde{\Xi}_{t,s}|_{\alpha,t}}{s^{-\eta_1}|s-t|^{\sigma}}<\infty\;, \qquad 
\llbracket  \tilde{\Xi} \rrbracket_{\mu;\eta_2;T}:=
\sum_{i\in J}
\sup_{\substack{0<s<u<t \leq T}} 
 \frac{
|\tilde{\Xi}^i_{tus}|_{\alpha-\rho_i,t}}{ |t-s|^{\mu} s^{-\eta_{2,i}}}<\infty \ .
$$
\end{definition}

We say that $\mathcal{P}$ is a partition of $[s,t]\subset \mathbb{R}$, if it is a finite collection of closed intervals such that the intersection of any two contains at most a point
and the union of all intervals in $\mathcal{P}$ equals $[s,t]$. We set $|\mathcal{P}|$ to be the length of the largest interval.
\begin{lemma}\label{lem:sewing}
For $0<\sigma\leq 1<\mu$, $\eta_1, \eta_{2,i}\geq 0$  and $\rho_i \in [0,1)$
let $\Xi\in C^{\sigma,\mu,\eta}_{2,\alpha,\rho}((0,T],V)$. Let $\beta\geq 0$ be such that $\beta+\rho_i<1$ for every $i\in J$.
\begin{enumerate}
\item\label{sewing_it1} Then,
 for every
$0<s<t\leq T$, the limit over partitions $\mathcal{P}$ of $[s,t]$
\[
(\mathcal I\Xi)_{t,s}
 := \lim_{|\mathcal P|\to 0}
      \sum_{[u,v]\in\mathcal P} S_{t-v}\Xi_{v,u}
      \in V_{\alpha,t}
\]
exists and satisfies $
(\mathcal I\Xi)_{t,s}
=
(\mathcal I\Xi)_{t,u}
+
S_{t-u}(\mathcal I\Xi)_{u,s}
$
for  $s<u<t$ and
\begin{equ}\label{sewing_error}
\bigl|(\mathcal I\Xi)_{t,s}-\Xi_{t,s}\bigr|_{\alpha+\beta,t}
 \lesssim
 \llbracket\tilde\Xi \rrbracket_{\mu;\eta_2;T}
 \sum_i s^{-\eta_{2,i}} |t-s|^{\mu-\rho_i-\beta}
\end{equ}
\item\label{sewing_it2} Assume in addition that $\sigma>\eta_1$ and $\mu>\eta_{2,i}+\rho_{i}$ for each $i$.
Then, for every $t\in(0,T]$, the limit
$
(\mathcal I\Xi)_{t,0}
:=
\lim_{s\downarrow0}(\mathcal I\Xi)_{t,s} \in V_{\alpha+\beta,t}
$
exists and
\begin{equ}\label{sewing_error2}
\bigl|(\mathcal I\Xi)_{t,0}\bigr|_{\alpha+\beta,t}
\lesssim
\llbracket  \tilde{\Xi}\rrbracket_{\sigma;\eta_1;T}
t^{\sigma-\beta-\eta_1}
+\sum_i
\llbracket  \tilde{\Xi} \rrbracket_{\mu;\eta_2;T}
t^{\mu-\rho_i-\beta-\eta_{2,i}}.
\end{equ}
\end{enumerate}
The implicit constant in both \eqref{sewing_error} and \eqref{sewing_error2} is independent of $t,s\in (0,T]$ and $\Xi\in C^{\sigma,\mu,\eta}_{2,\alpha,\rho}((0,T],V)$.
\end{lemma}
\begin{proof}
 Fix $s<t\leq T$ and define partitions $\mathcal{P}_0=\{[s,t]\}$ and recursively by $\mathcal{P}_{n+1}= \{[u,(u+v)/2] \ : [u,v]\in \mathcal{P}_{n} \}  \cup\{[(u+v)/2,v] \ : [u,v]\in P_{n} \} $.
 We set $\mathcal{I}^{\mathcal{P}}\Xi_{t,s}= \sum_{[u,v]\in \mathcal{P}} S_{t-v}\Xi_{v,u}$.
 Then 
 $$\mathcal{I}^{\mathcal{P}_{n+1}}\Xi_{t,s}-\mathcal{I}^{\mathcal{P}_{n}}\Xi_{t,s}=- \sum_{[u,v]\in \mathcal{P}^n} S_{t-v} (\hat{\delta}\Xi)_{v, (v+u)/2,u} \ .
$$ 
Since
$S_{t-v} (\hat{\delta}\Xi)_{v, m,u}= \sum_i S_{t-m} \tilde{\Xi}^i_{v, m,u}$
\begin{align*}
|S_{t-m} \tilde{\Xi}_{v, m,u}|_{\alpha+\beta,t}&\lesssim\sum_i |t-m|^{-\beta-\rho_i} |\tilde{\Xi}^{i}_{v, m,u}|_{\alpha-\rho_i,t}\\
&\lesssim
 \llbracket\tilde\Xi \rrbracket_{\mu;\eta_2;T}
\sum_i |t-m|^{-\beta-\rho_i} |v-u|^{\mu} s^{-\eta_{2,i}}  \ .
\end{align*}
Note that $\mathcal{P}^{n}= \{s+ [k,k+1] \frac{t-s}{2^{n}} \ : k=0,..., 2^{n}-1 \}$ and therefore 
\begin{align*}
\sum_{[u,v]\in \mathcal{P}^n}| S_{t-v} (\hat{\delta}\Xi)_{v, (v+u)/2,u}|_{\alpha+\beta,t}
&\lesssim \sum_i s^{-\eta_{2,i}} \llbracket  \tilde{\Xi} \rrbracket_{\mu;\eta_2;T} \Big(\frac{t-s}{2^{n} }\Big)^{\mu}\sum_{k=0}^{2^{n}-1}
\Big|t-s-(k+1/2)\frac{t-s}{2^{n}}\Big|^{-\beta-\rho_i}  \\
& \lesssim \sum_i s^{-\eta_{2,i}} |t-s|^{\mu-\beta-\rho_i} \llbracket  \tilde{\Xi} \rrbracket_{\mu;\eta_2;T} \Big(\frac{1}{2^{n} }\Big)^{\mu}\sum_{k=0}^{2^{n}-1}
\Big|1-\frac{k+1/2}{2^{n}}\Big|^{-\beta-\rho_i}  \\
& \lesssim \sum_i s^{-\eta_{2,i}} |t-s|^{\mu-\beta-\rho_i} \llbracket  \tilde{\Xi} \rrbracket_{\mu;\eta_2;T} 2^{-n(\mu-1)}
\end{align*}
where we have used that $\rho_i+\beta<1$ in the last inequality. Since $\mu>1$ this is in turn summable and $(\mathcal{I}\Xi)_{t,s}:= \lim (\mathcal{I}^{\mathcal{P}_n}\Xi)_{t,s}$ satisfies \eqref{sewing_error}. The fact that this limit is independent of the sequence of partitions used and that $(\mathcal I\Xi)_{t,s}
=
(\mathcal I\Xi)_{t,u}
+
S_{t-u}(\mathcal I\Xi)_{u,s}$ follows exactly as in \cite{friz_hairer_20_introduction}.
Thus we conclude Item~\ref{sewing_it1} of the lemma.

For Item~\ref{sewing_it2},
let $t_n= 2^{-n}t$. Then\footnote{It follows that this is legitimate a posteriori from the bounds below.}
\begin{equs}
&\bigl|(\mathcal I\Xi)_{t,t_N}\bigr|_{\alpha+ \beta, t}
\leq \sum_{n=0}^{\infty} \bigl|S_{t-t_n}(\mathcal I\Xi)_{t_n,t_{n+1}}\bigr|_{\alpha+ \beta, t} \\
&\leq \bigl| (\mathcal I\Xi)_{t,t_{1}} - \Xi_{t,t_{1}} \bigr|_{\alpha+ \beta, t}+ \sum_{n=0}^{\infty}\bigl| S_{t-t_n} \Xi_{t_n,t_{n+1}}\bigr|_{\alpha+ \beta, t} 
+\sum_{n\geq 1}  \bigl| S_{t-t_n}\big((\mathcal I\Xi)_{t_n,t_{n+1}} - \Xi_{t_n,t_{n+1}}\big) \bigr|_{\alpha+ \beta, t}\label{eq:localsum}
\end{equs}
For the first term we have 
$$
\bigl| (\mathcal I\Xi)_{t,t_{1}} - \Xi_{t,t_{1}} \bigr|_{\alpha+ \beta, t}
\lesssim  \llbracket\tilde\Xi \rrbracket_{\mu;\eta_2;T}
\sum_i t^{\mu-\rho_i-\beta-\eta_{2,i}} \ .
$$
To estimate the first sum, note that for $n\geq 0$
$$\bigl|S_{t-t_n}\Xi_{t_n,t_{n+1}}\bigr|_{\alpha+\beta, t}= \bigl|S_{t-t_{n+1}} {\Xi}_{t_n,t_{n+1}}\bigr|_{\alpha+\beta, t}
\leq \llbracket\tilde{\Xi}\rrbracket_{\sigma;\eta_1;T}  (t-t_{n+1})^{-\beta}  t_n^{\sigma-\eta_1}
\lesssim t^{-\beta}  \llbracket\tilde{\Xi}\rrbracket_{\sigma;\eta_1;T}   t_n^{\sigma-\eta_1} \ .
$$
To bound the last sum, note that for $n\geq 1$ (since then $t\lesssim t-t_n$)
\begin{align*}
\bigl|
 S_{t-t_n}\big(
(\mathcal I\Xi)_{t_n,t_{n+1}}-\Xi_{t_n,t_{n+1}}\big)\bigr|_{\alpha+\beta,t}
 \lesssim
 t^{-\beta} \bigl|(\mathcal I\Xi)_{t_n,t_{n+1}}-\Xi_{t_n, t_{n+1}}\bigr|_{\alpha,t} 
 \lesssim  t^{-\beta} 
 \llbracket\tilde\Xi \rrbracket_{\mu;\eta_2;T} 
 \sum_{i} 
 t_{n+1}^{\mu-\eta_{2,i}-\rho_i} \ .
\end{align*}
By \eqref{eq:localsum} we conclude by summing over $n\geq 0$, respectively $n\geq 1$ that 
$
(\mathcal I\Xi)_{t,0}:=\lim_{n\to \infty}(\mathcal I\Xi)_{t,t_n}  \in V_{\alpha+\beta,t}
$
exists and satisfies the desired bound. That one can replace $t_n$ by any sequence converging to $0$ follows by a similar argument.
 \end{proof}
 
\begin{remark}
Let us note that allowing $\hat{\delta}$ to have contributions which behave differently, as in the above Definition is crucial when working with truly infinite-dimensional noises.
In particular \cite[Thm.~2.4]{GerasimovicsHairer}, which is also a reinterpretation of \cite{Gub_tindel_sewing}, is not sufficient in the setting of our article, even if we work
on a compact space (so that one Banach space is sufficient) and with regular initial conditions.
\end{remark}

\subsection{A Fixed-Point Theorem}\label{sec:fixed point}

We write \eqref{eq:illustrate} in mild form 
\begin{equ}\label{eq:mildSPDE}
u(t) = P_t u_0+ \sum_{d(I)<m} c_I \int_0^t  P_{t-s} (X^I u_s dW_s^{(I)}  ) 
\end{equ}
where $P_t$ is the semigroup associated to the $m$-th order operator $ \mathcal{R}$, see Remark~\ref{rem:heat_semi}, and $\partial_t W^{(I)}_t  =\xi^{(I)}$.
We fix the weights $e_\ell$, $p_a$ as in \eqref{eq:weights_examle}  and set $w_t:=e_{\ell+t}$. To simplify exponents, 
set $V_{\alpha,t}:= B^{m\alpha, w_t}_{p, \infty}$ for some fixed $p\geq 1$ and  $\| \cdot\|_{\alpha, t}:= \| \cdot \|_{B^{m\alpha,w_t}_{p,\infty}}$. Note that \eqref{eq:semi-group_bounds} follows from Proposition~\ref{prop:schauder}.
We shall construct solutions in the Banach space $\mathcal{B}_T^{\gamma,r, \theta}$ of processes $ (0,T] \ni t\mapsto v_t \in V_{r,t}$, which we equip with the norm
$$\vertiii{v}_{r, \gamma, \theta}:= \sup_{t\in (0,T]} \frac{\|v_t\|_{r, t}}{t^{-\theta}} +\sup_{
\substack{
s<t\in (0,T]\\
|t-s|<s
}}  \frac{\| v_t-v_s\|_{r, t}}{s^{-\theta-\gamma} |t-s|^\gamma} 
<\infty \ .
$$

\begin{theorem}\label{prop:fixedpoint}
Fix \(p\in[1,\infty]\), \(T>0\), \(h\in(0,1]\), $a\in [0,h)$, and
$\alpha_I\in \mathbb{R}$, 
$
\mathcal{K}:=\{K\ :\ c_K\neq 0\}\ 
$
and assume  $\alpha:=\min_{K\in\mathcal{K}} \alpha_K\leq 0$.
  Assume that 
\begin{equ}\label{conditions}
2(h-a) + \min \Big\{ \alpha+ \min_{K\in \mathcal{K}} \big(\alpha_K - d(K)/m \big) ,\ -\max_{K\in \mathcal{K}} \frac{d(K)}{m}  \Big\} >1 \ .
\end{equ}
For $\nu\in \big[0,h-a+\min_{K\in \mathcal{K}} (\alpha_K - d(K)/m )\big)$, let $r\in \mathbb{R}_+$ and $\beta_K\geq 0$ and $\gamma>0$ be such that\footnote{These conditions are satisfiable by \eqref{conditions}.}
\begin{equ}\label{eq:fixed_point_exponents}
-\min_{K\in \mathcal{K}} \big(\alpha_K - d(K)/m \big) <r <\big(\alpha +2(h-a) -1\big) \wedge (h-a-\nu) , \quad  1-(h-a)<\gamma<h-a-\max \beta_K \ 
\end{equ}
and for all $K\in \mathcal{K}$
\begin{equ}\label{eq:fixed_point_exponents2}
\frac{d(K)}{m}\vee (r-\alpha_K) <\beta_K<2(h-a)-1  \ .
\end{equ}
Then, for
 any 
$u_0\in V_{-\nu,0}$ and $W_t^{I}\in C^{h}([0,T], B^{m\alpha_I,p_a}_{\infty,\infty})$ 
equation \eqref{eq:mildSPDE} admits a unique solution
$u\in \mathcal{B}_T^{\gamma,r,r+\nu}$.
Furthermore, the solution map $$\prod_{d(I)<m} C^{h}([0,T], B^{m\alpha_I,p_a}_{\infty,\infty}) \times V_{-\nu,0} \to \mathcal{B}_T^{\gamma,r,r+\nu}, \qquad (W,u_0)\mapsto u$$
is Lipschitz in the initial condition $u_0$ and 
locally Lipschitz in the noise $W$.
\end{theorem}

\begin{remark}
Let us comment on the numerical values in Theorem~\ref{prop:fixedpoint}. First consider the case that only $c_0\neq 0$. 
In the proof below we will see that the fact that a noise of time regularity $h$ belongs to a $p_a$-weighted space results in the `effective' time regularity of $W$ being
$h-a$. Thus the effective space-time regularity of $\xi= \partial_t W$ is $\alpha^{\fraks}:= m(h-a -1) + m\alpha$. Inserting this into \eqref{eq:young_regime} exactly gives the first inequality of \eqref{conditions}, since $\max_{K\in \mathcal{K}}\frac{d(K)}{m}=0$.

Similarly, in the general case the first term in the minimum of \eqref{conditions} guarantees that all products $X^{I}u \cdot W^{(I)}$ can be made sense of as Young products.

Finally, the second term in the minimum  of \eqref{conditions} arises due to the time-sewing step. Recalling that after `trading' time-regularity for space-regularity, 
$X^{I}u$ has time-regularity  $(h-a)-\max_{K\in \mathcal{K}} \frac{d(K)}{m}$, while $\Xi^{(I)}$ has time regularity $h-a$
resulting in the 
condition $2(h-a)-\max_{K\in \mathcal{K}} \frac{d(K)}{m}>1$.
\end{remark}
\begin{proof}
We apply the Banach fixed-point theorem in the space $B_T$ for $T>0$ sufficiently small,
to the map
\begin{equ}
\Phi(u)(t):=P_tu_0+\sum_{d(I)<m}c_{I}\Gamma_I(u)(t),
\qquad
\Gamma_I(u)(t):=\int_0^t P_{t-s}(X^Iu_s \cdot dW_s^{(I)})  \label{fix point}
\end{equ}
where the latter is obtained as a Young integral using Lemma~\ref{lem:sewing}. 
We first estimate the evolution of the initial condition.
Assume that $u_0\in V_{-\nu,0}$, then by Proposition~\ref{prop:schauder}
$$\|P_t u_0\|_{ {r},t} 
\lesssim t^{-{r}-\nu} \|u_0\|_{-\nu,t} \lesssim t^{- {r}-\nu} \|u_0\|_{-\nu,0} $$
as well as 
\begin{align*}
\|P_s u_0- P_t u_0\|_{ {r},t}&= \|(\id -P_{t-s})P_s u_0\|_{ {r},t} 
\lesssim |t-s|^{\gamma}\|P_s u_0\|_{{ {r}+\gamma, t}} \\
&
\lesssim|t-s|^{\gamma} s^{- {r}-\gamma-\nu}\|u_0\|_{-\nu,0}
\end{align*}
Thus we see that $P_t u_0\in \mathcal{B}$, due to the choice $\theta= {r}+\nu$.

Next we estimate the integral terms by applying Lemma~\ref{lem:sewing}. To lighten notation we simply write $\vertiii{W }$ without any subscript to denote maximum over $I$ of the 
$C^{h}([0,T], B^{m\alpha_I,p_a}_{\infty,\infty}) $-norm  of $W^{(I)}$. We write $\beta=\max_{K\in \mathcal{K}}\beta_K$.
Then, set
\begin{equ}\label{eq:increment}
\Xi^{(I)}_{t,s}:=P_{t-s}(X^{I}u_s \cdot W^{(I)}_{t,s}) \qquad \text{for} \qquad W^{(I)}_{t,s}:=W^{(I)}_{t}-W^{(I)}_s\ .
\end{equ}
We have the bound
\begin{equs}
|\tilde{\Xi}^{(I)}_{t,s} |_{r-\beta_I,t}&= \|X^{I}u_s \cdot (W^{(I)}_t-W^{(I)}_s )\|_{r-\beta_I,t}
\lesssim |t-s|^{-a} \|X^{I}u_s \cdot (W^{(I)}_t-W^{(I)}_s )\|_{B^{m(r-\beta_I),p_aw_s}_{p,\infty}}\\
&\lesssim   |t-s|^{h-a}  s^{-\theta}    \vertiii{u}_{ {r}, \gamma, \theta}\vertiii{W}\ \label{fixpoint_ineq}
\end{equs}
where in the first inequality we used \eqref{weights_trade} and in the second one Young's multiplication theorem, Theorem~\ref{prop:young}, where 
 \eqref{eq:fixed_point_exponents} guarantees the exponents are in the right range.

Note that for
$$ \tilde{\Xi}^{{(I)},1}_{tus}:=
 (P_{u-s}-\id) (X^{I}u_s \cdot W^{(I)}_{t,u}),
\qquad \tilde{\Xi}^{{(I)},2}_{tus}:= (X^{I}u_s-X^{I}u_u)\cdot W^{(I)}_{t,u}, 
$$
we have
\begin{align*}
\hat{\delta}\Xi^{(I)}_{tus}&= P_{t-s}(X^{I}u_s \cdot W^{(I)}_{t,s} ) - P_{t-u}(X^{I}u_u \cdot W^{(I)}_{t,u})- P_{t-s}(X^{I}u_s \cdot W^{(I)}_{u,s})  \\
&= P_{t-u} \big( P_{u-s}(X^{I}u_s \cdot W^{(I)}_{t,u})-X^{I}u_u\cdot  W^{(I)}_{t,u}\big)=  P_{t-u}(\tilde{\Xi}^{(I),1}_{tus}+\tilde{\Xi}^{(I),2}_{tus}) \ .
\end{align*}

We shall use an auxiliary parameter
$\kappa\in \big(1-(h-a), \min\{\gamma,1-\beta-\gamma\}\big)$, which exists by the first inequality of \eqref{eq:fixed_point_exponents2}.

Furthermore, similarly to \eqref{fixpoint_ineq}
\begin{align*}
| \tilde{\Xi}^{(I),1}_{tus} |_{r-\beta_I-\kappa,t}&=
\| (P_{u-s}-\id) (X^{I}u_s \cdot W^{(I)}_{t,u}) \|_{r-\beta_I-\kappa,t}\\
&\lesssim |u-s|^\kappa \| (X^{I}u_s \cdot W^{(I)}_{t,u}) \|_{r-\beta_I,t} \\
&\lesssim |u-s|^\kappa |t-s|^{h-a} s^{-\theta} \vertiii{u}_{ {r}, \gamma, \theta} \vertiii{W}\\
&\lesssim |t-s|^{\kappa+h-a} s^{-\theta} \vertiii{u}_{ {r}, \kappa, \theta} \vertiii{W}
\end{align*}
and similarly
\begin{align*}
| \tilde{\Xi}^{(I),2}_{tus} |_{r-\beta_I,t}&=
\| (X^{I}u_s-X^{I}u_u) \cdot W^{(I)}_{t,u} \|_{r-\beta_I,t}\\
&\lesssim |u-s|^\kappa |t-u|^{h-a}s^{-\theta-\kappa} \vertiii{u}_{ {r}, \kappa, \theta} \vertiii{W}\\
&\lesssim |t-s|^{\kappa +h-a}s^{-\theta-\kappa} \vertiii{u}_{{r}, \gamma, \theta} \vertiii{W}
\end{align*}
where we used that 
$
\|X^{I}u_u-X^{I}u_s\|_{ {r}-d(I)/m,u}
\lesssim
\vertiii{u}_{ {r},\gamma,\theta}
s^{-\theta-\kappa}|u-s|^\kappa.
$
Thus, we next check the assumptions on the numerical exponents of the sewing lemma, Lemma~\ref{lem:sewing}, which we list for the reader's convenience.

%


\begin{table}[H]
\centering
\begin{tabular}{|c|c|c|c|c|c|c|c|}
\hline
\textbf{Sewing} 
& $\sigma$ 
& $\mu$ 
& $\rho_1$ 
& $\rho_2$ 
& $\eta_1$ 
& $\eta_{2,1}$ 
& $\eta_{2,2}$ \\
\hline
\textbf{SPDE} 
& $h-a$ 
& $\kappa+h-a$ 
& $\kappa$ 
& $0$ 
& $\theta= {r}+\nu$ 
& $\theta= {r}+\nu$ 
& $\theta+\kappa= {r}+\nu+\kappa$ \\
\hline
\end{tabular}
\end{table}

We have to check the following conditions to apply the sewing lemma
$$0<\sigma\leq 1 < \mu, \qquad \beta+\rho_i<1, \qquad \sigma-\eta_1>0, \qquad \mu-\rho_i-\eta_{2,i}>0\ .$$
The first inequalities for $\sigma$ are satisfied by the assumed range of $h$ and $a$. The first inequalities for $\mu$ is satisfied by the choice of $\kappa$. The second condition for $i=1$ holds again by the choice of $\kappa$, while the case $i=2$ is clear. The third condition holds by the first inequality of \eqref{eq:fixed_point_exponents}. The fourth condition agrees for both $i\in \{1,2\}$ with the third.

Thus by Lemma~\ref{lem:sewing} the right-hand side of \eqref{fix point} is well defined by setting
$
\Gamma_I(u)(t):=
(\mathcal I\Xi^{(I)})_{t,0} \ .
$
Furthermore, by \eqref{sewing_error2}
\begin{equs}
\|\Gamma_I(u)(t)\|_{ {r},t}
&\lesssim \Big(t^{\sigma-\beta-\eta_1} + \sum_{i} t^{\mu-\rho_i-\beta-\eta_{2,i}} \Big)\vertiii{u}_{ {r},\gamma,\theta} \vertiii{W}
\lesssim
t^{h-a-\beta-\theta} \vertiii{u}_{ {r},\gamma,\theta} \vertiii{W} \ ,
\end{equs}
and thus \begin{equation}\label{eq:Gamma-size-fixedpoint}
\sup_{0<t\leq T}
t^\theta\|\Gamma_I(u)(t)\|_{ {r},t}
\lesssim
T^{h-a-\beta} \vertiii{u}_{ {r},\gamma,\theta} \vertiii{W} \ .
\end{equation}
To estimate the increment, write
$
\Gamma_I(u)(t)-\Gamma_I(u)(s)
=
(\mathcal I\Xi^{(I)})_{t,s}
+
(P_{t-s}-\id)\Gamma_I(u)(s)
$,
we bound the first term using that 
\begin{align}
\|(\mathcal I\Xi^{(I)})_{t,s}
  -\Xi^{(I)}_{t,s}\|_{ {r},t}
&\lesssim
\Bigl(
s^{-\theta}|t-s|^{h-a-\beta}
+
s^{-\theta-\kappa}|t-s|^{h-a+\kappa-\beta}
\Bigr) \vertiii{u}_{ {r},\gamma,\theta} \vertiii{W},
\end{align}
together with \eqref{fixpoint_ineq} to obtain that 
\begin{align*}
\|(\mathcal I\Xi^{(I)})_{t,s}\|_{ {r},t}
&\lesssim
\Bigl(
s^{-\theta}|t-s|^{h-a-\beta}
+
s^{-\theta-\kappa}|t-s|^{h-a+\kappa-\beta}
\Bigr) \vertiii{u}_{ {r},\gamma,\theta} \vertiii{W}\\
&
\lesssim
T^{h-a-\beta}
s^{-\theta-\gamma}|t-s|^\gamma  \vertiii{u}_{ {r},\gamma,\theta} \vertiii{W}\ .
\end{align*}
For the remaining terms using \eqref{sewing_error2} in the second inequality gives
\begin{align*}
\|(P_{t-s}-\id)\Gamma_I(u)(s)\|_{ {r},t}
&\lesssim
|t-s|^\gamma
\|\Gamma_I(u)(s)\|_{ {r}+\gamma,s}\\
&\lesssim
|t-s|^\gamma
s^{h-a-\beta-\gamma-\theta} \vertiii{u}_{ {r},\gamma,\theta} \vertiii{W}\\
&\lesssim
T^{h-a-\beta}
s^{-\theta-\gamma}|t-s|^\gamma \vertiii{u}_{ {r},\gamma,\theta} \vertiii{W} \ ,
\end{align*}
which together with \eqref{eq:Gamma-size-fixedpoint} proves that 
$
\vertiii{\Gamma_I(u)}_{ {r},\gamma,\theta}
\lesssim
T^{h-a-\beta}
\vertiii{u}_{ {r},\gamma,\theta} \vertiii{W} .
$

Since $h-a-\beta>0$, the map
$\Phi$ is a contraction on
$\mathcal B_\tau^{\gamma, {r},r+\nu}$ for 
$\tau>0$ sufficiently small, depending on $\vertiii{W}$ but not the norm of the initial condition.
 Restarting the equation a finite number of times gives 
a unique solution on $(0,T]$.
The proof of the continuity claim follows by standard arguments.
\end{proof}

\subsection{The Cole--Hopf transform for Rockland operators}
Throughout this section we fix a Rockland operator $\mathcal{R}$ of
homogeneous degree $m$ and reserve $r_I\in\mathbb{R}$ for its coefficients as in \eqref{eq:form rockland}.
We next study how the Cole--Hopf transform behaves for Rockland operators. 
For $I\neq 0$, let $\mathbf{w}(I)$ be the word containing $i_1$ copies of $1$,
followed by $i_2$ copies of $2$, and so on. If $A$ is a subset of these letters, $I_A$ denotes the
multi-index of the corresponding subword, with the induced order.
We write $P(I)$ for the set of partitions of the letters of 
${\mathbf{w}}(I)$ and define
\begin{equation}\label{eq:bell-polynomial}
 B_0(v):=1,\qquad
 B_I(v):=e^{-v}X^I(e^{v})
 = \sum_{n\geq 1} B^{(n)}_I(v), \qquad 
 B^{(n)}_I(v):=
 \sum_{\substack{\pi\in P(I)\\ |\pi|=n} }
       \prod_{A\in\pi}X^{I_A}v .
\end{equation}
Next set for $d(K)<m$
  $$A_K(v):=\sum_{n\geq 1} A^{(n)}_K(v), \qquad   
  A^{(n)}_K(v):=
  \sum_{\substack{I\geq K\\d(I)=m}}
      r_I\binom IK  B^{(n)}_{I-K}(v)
      $$
     and note that $ A_0(v)=\mathcal{R}v+\mathcal Z_{\mathcal{R}}(v)$ for 
     \begin{equation}
 \mathcal Z_R(v):=
 \sum_{n\geq 2}
 \mathcal Z^{n}_R(v),
 \qquad 
 \mathcal Z^{(n)}_R(v):=
 \sum_{d(I)=m}r_I
 \sum_{\substack{\pi\in  P(I)\\|\pi|=n}}
\prod_{A\in\pi}X^{I_A}v . 
 \label{eq:ZR-general}
\end{equation}

\begin{lemma}\label{lem:coleHopf}
Assume $v,\xi,f\in C^{\infty} (\mathbb{G})$ satisfy
$
  \mathcal{R}v=\xi+f.
$
If $u$ satisfies
\begin{equation}
  (\partial_t+  \mathcal{R})u
   =u(\xi-C) \ ,
  \label{eq:pam-epsilon-section33}
\end{equation}
then $w:=e^{-v}u$ satisfies
\begin{equation}\label{transformed_eq}
 (\partial_t+\mathcal R)w
 =
 -\sum_{0<d(K)<m}A_K(v)X^Kw
 -\bigl(\mathcal Z_{\mathcal R}(v)+f+C\bigr)w.
\end{equation}

\end{lemma}
\begin{proof}
Repeated use of the Leibniz rule together with the definitions in the preamble of the lemma show that
\begin{equs}\label{eq:conjugated-rockland}
 e^{-v}\mathcal   \mathcal{R}(e^{v}w)
 &=\mathcal{R}w+ \sum_{d(K)< m}A_K(v)X^Kw
 &=  \mathcal Rw  +A_0(v) w+  \sum_{0<d(K)< m}A_K(v)X^Kw \ .
\end{equs}
Since on the other hand 
$\partial_t w +e^{-v} \mathcal{R}(e^{v} w)= e^{-v} (\partial_t+  \mathcal{R})u = e^{-v} u(\xi-C)=w (\xi-C)$
we conclude that 
$$
\partial_t w + \mathcal{R}( w)=   w (\xi-C-A_0(v))-  \sum_{0<d(K)< m}A_K(v)X^K w 
$$
which agrees with \eqref{transformed_eq} since $
A_0(v)=  \mathcal{R}v+\mathcal Z_ \mathcal{R}(v)= \xi+f+\mathcal Z_\mathcal{R}(v) \ .
$
\end{proof}

\begin{lemma}
\label{prop:only-scalar-counterterm}
Recalling $\delta_R$ from \eqref{eq:delta-R}, let $s>0$ be such that 
$
  2s>m-\delta_R
 $. Let $a\in \mathbb{R}$ and $\eta>0$.
\begin{enumerate}
 \item For every $K\neq 0$ and $n\geq 1$ the following is a locally uniformly continuous map
 \begin{equation}
  B_{\infty,\infty}^{s,p_a} \to  B_{\infty,\infty}^{s-m+d(K)+(n-1)\delta_R-\eta,p_{na}}, \qquad v\mapsto A^{(n)}_{K}(v) \ .
  \label{eq:AK-nonlinear-regularity1}
 \end{equation}
 \item For every $n\geq 3$ the following is a locally uniformly continuous map
  \begin{equation}
  B_{\infty,\infty}^{s,p_a}\to  B_{\infty,\infty}^{s-m+(n-1)\delta_R-\eta,p_{na}}, \qquad v\mapsto \mathcal{Z}_{\mathcal{R}}^{(n)}(v) \ .
  \label{eq:AK-nonlinear-regularity2}
 \end{equation}
\end{enumerate}
\end{lemma}
\begin{proof}
Each term contributing to $A^{(n)}_K(v)$ in \eqref{eq:bell-polynomial} arises from a partition
$\pi=\{A_1,\ldots,A_n\}$. Since for each term contributing $d(I_{A_j})\geq\delta_{\mathcal{R}}$ and 
\begin{equation}\label{eq:degree-sum-AK}
 \sum_{j=1}^n d(I_{A_j})=m-d(K),
\end{equation} and 
for two distinct blocks in a term in $A_K$, using $K\neq 0$ and
\eqref{eq:degree-sum-AK},
\[
 2s-d(I_{A_j})-d(I_{A_k})
 \geq2s-m+d(K)+(n-2)\delta_{\mathcal R}>0.
\]
Thus \eqref{eq:AK-nonlinear-regularity1} makes sense as a Young product by Theorem~\ref{prop:young}, where the $\eta>0$ is present in order to exclude regularities belonging to $\triangle$.

The argument for \eqref{eq:AK-nonlinear-regularity2} follows similarly.
\end{proof}

\subsubsection{The proof of Theorem~\ref{thm:first singular}}\label{sec:proof_of_singularPAM}
Denote by $G$ the fundamental solution of the operator $1+\mathcal R$ and set
$$
v_\varepsilon=\xi_\varepsilon*G, \qquad C_\varepsilon
 :=
\mathbb -\E \left[
\mathcal Z_{\mathcal R}^{(2)}(v_\varepsilon)(e) \right] .
$$
We shall obtain the solution in $\mathcal{B}_T^{\gamma,r, r+\nu}$
where the parameters $\gamma,r$ will be chosen according to \eqref{eq:fixed_point_exponents} and \eqref{eq:fixed_point_exponents2} depending on $\zeta$.

\begin{proof}
Note that $\mathcal Rv_\varepsilon=\xi_\varepsilon-v_\varepsilon$ and 
$
v_\varepsilon\to v:=\xi*G
$
in $
B_{\infty,\infty}^{\zeta+m-\kappa,p_a}
$
for every \(a,\kappa>0\).
Instead of solving for $u_\eps$ directly, we consider $w_\eps=e^{-v_\eps}u_\eps$ as in Lemma~\ref{lem:coleHopf} which solves
\begin{equation}
(\partial_t+\mathcal R)w_\varepsilon
 =
\sum_{0<d(K)<m}
\xi_\varepsilon^{(K)}X^Kw_\varepsilon
+\xi_\varepsilon^{(0)}w_\varepsilon,
\qquad
w_\varepsilon(0)=e^{-v_\varepsilon}u_0.
\label{eq:first-singular-transformed}
\end{equation}
for
$$
\xi_\varepsilon^{(0)}
 :=
-\widehat{\mathcal Z}_\varepsilon
 -\sum_{n\geq 3}\mathcal Z_{\mathcal R}^{(n)}(v_\varepsilon)
 +v_\varepsilon,
\qquad
\xi_\varepsilon^{(K)}
 :=
-A_K(v_\varepsilon),\quad K\neq 0,
$$
and 
$
\widehat{\mathcal Z}_\varepsilon
 :=
\mathcal Z_{\mathcal R}^{(2)}(v_\varepsilon)+C_\varepsilon \ .
$
We postpone the proof of the following claim to the end of this section.
\begin{claim}\label{claim}
For every $a,\kappa>0$ and 
$
\tilde{\alpha}_0:=\min\{2\zeta+m,\zeta+\delta_{\mathcal{R}}\}$ and $
\tilde{\alpha}_K:=\zeta+d(K)$ for $K\neq0 \ ,$ it holds that 
\begin{equation}
\xi_\varepsilon^{(K)}
\to \xi^{(K)}
\label{eq:first-singular-coefficients}
\end{equation}
in $B_{\infty,\infty}^{\tilde{\alpha}_K-\kappa,p_a}$ in probability.
\end{claim}
Thus, we shall apply Theorem~\ref{prop:fixedpoint} to \eqref{eq:first-singular-transformed} with $W^{(K)}_\eps(t):=t\xi^{(K)}_\eps$.
We first check that the regularity assumptions are satisfied.
For the numerical check: for $\alpha_K:= \frac{\tilde{\alpha}_K-\kappa}{m}$, $\alpha= \min_{K\in \mathcal{K}} {\alpha_K}$ and $h=1$ the condition \eqref{conditions} follows from
\begin{equ}\label{conditions2}
 \min_{{J}\in \mathcal{K}}\tilde{\alpha}_J+ \min_{K\in \mathcal{K}} \big(\tilde{\alpha_K} - d(K) \big) >-m(1-2a)+2\kappa \ .
\end{equ}
Since every $K\in \mathcal{K}\setminus\{0\}$ satisfies $d(K)\geq \delta_{\mathcal{R}}$ we have  $\min_{{J}\in \mathcal{K}}\tilde{\alpha}_J=\tilde{\alpha}_0$ and 
$\min_{K\in \mathcal{K}} \big(\tilde{\alpha_K} - d(K) \big)\geq\zeta$
Thus \eqref{conditions2} follows from
$\tilde{\alpha}_0 + \zeta= 
\min\{3\zeta+m,2\zeta+\delta_{\mathcal{R}}\}
>-m(1-2a)+2\kappa, $
which is satisfied by choosing $a,\kappa$ sufficiently small.

To conclude the proof,  the convergence $v_\varepsilon\to v:=\xi*G$,
Theorem~\ref{th:taylor} and the Faà di Bruno's formula imply that
\begin{equ}\label{convergence_of_exp}
e^{-v_\eps}\to e^{-v}, \qquad e^{v_\eps}\to e^{v}
\end{equ}
 in $B_{\infty,\infty}^{\zeta+m-\kappa,e_\eta}$. 
This implies that the initial condition $w_\eps(0)= e^{-v_\eps}u_0\to  e^{-v}u_0=w_0(0)$ 
 indeed converges in $B_{p,\infty}^{-m\nu, e_\ell}$ for $\ell$ sufficiently large by Theorem~\ref{prop:young}.
 Again, by \eqref{convergence_of_exp} and Theorem~\ref{prop:young} we conclude that $u_\eps=e^{v_\eps}w_\eps\to e^{v}w$ in the same space (up to increasing the weight again). Locally uniform convergence follows by repeatedly 
using the Besov embedding Corollary~\ref{prop:embedding} Item~\ref{Item_with_num_cond} and restarting the equation.
\end{proof}

\begin{proof}[Proof of Claim~\ref{claim}]
The claim for $K\neq 0$ follows directly from the fact that $
v_\varepsilon\to v
$
in $
B_{\infty,\infty}^{\zeta+m-\kappa,p_a}
$
combined with Lemma~\ref{prop:only-scalar-counterterm}. For the same reason the only term in 
$\xi_\varepsilon^{(0)}
=
-\widehat{\mathcal Z}_\varepsilon
 -\sum_{n\geq 3}\mathcal Z_{\mathcal R}^{(n)}(v_\varepsilon)
 +v_\varepsilon,$ which requires treatment is 
 $\widehat{\mathcal Z}_\varepsilon$.

Since this is a very standard argument, we stay brief and refer to \cite{hairer_labbe_15_simple} and \cite{MS25} for more details. This term is a finite linear combination of standard quadratic Wick products with
covariances which are bounded by a multiple of
$
1+|z|^{4\zeta+2m-\kappa},
$
which is locally integrable. 
Gaussian hypercontractivity and the
weighted Kolmogorov criterion, Lemma~\ref{lem:kolmogorov}, then show that 
$
\widehat{\mathcal Z}_\varepsilon
 \to \widehat{\mathcal Z}$ in
$B_{\infty,\infty}^{2\zeta+m-\kappa,p_a}$
in probability.
\end{proof}

	\bibliographystyle{Martin}
	\bibliography{HomGroups.bib}

\newcommand{\etalchar}[1]{$^{#1}$}
\begin{thebibliography}{HRRvV25}
\def\myhref#1#2{\href{#2}{\nolinkurl{#1}}}

\bibitem[Ant22]{mouzard_22_weyl}
\textsc{M.~Antoine}.
\newblock Weyl law for the {A}nderson {H}amiltonian on a two-dimensional
  manifold.
\newblock \emph{Ann. Inst. Henri Poincar\'e{} Probab. Stat.} \textbf{58},
  no.~3, (2022), 1385--1425.
\newblock
  \myhref{doi:10.1214/21-aihp1216}{https://doi.org/10.1214/21-aihp1216}.

\bibitem[BB26]{BB16}
\textsc{I.~Bailleul} and \textsc{Y.~Bruned}.
\newblock Locality for singular stochastic {PDE}s.
\newblock \emph{Ann. Probab.} \textbf{54}, no.~2, (2026), 768--795.
\newblock \myhref{doi:10.1214/25-aop1778}{https://doi.org/10.1214/25-aop1778}.

\bibitem[BCCH21]{BCCH20}
\textsc{Y.~Bruned}, \textsc{A.~Chandra}, \textsc{I.~Chevyrev}, and
  \textsc{M.~Hairer}.
\newblock Renormalising {SPDE}s in regularity structures.
\newblock \emph{J. Eur. Math. Soc. (JEMS)} \textbf{23}, no.~3, (2021),
  869--947.
\newblock \myhref{doi:10.4171/jems/1025}{https://doi.org/10.4171/jems/1025}.

\bibitem[BCD18]{Tindel_ref1}
\textsc{H.~Bahouri}, \textsc{J.-Y. Chemin}, and \textsc{R.~Danchin}.
\newblock Tempered distributions and {F}ourier transform on the {H}eisenberg
  group.
\newblock \emph{Ann. H. Lebesgue} \textbf{1}, (2018), 1--46.
\newblock \myhref{doi:10.5802/ahl.1}{https://doi.org/10.5802/ahl.1}.

\bibitem[BCH{\etalchar{+}}24]{fractals_1}
\textsc{F.~Baudoin}, \textsc{L.~Chen}, \textsc{C.-H. Huang},
  \textsc{C.~Ouyang}, \textsc{S.~Tindel}, and \textsc{J.~Wang}.
\newblock {Parabolic Anderson model in bounded domains of recurrent metric
  measure spaces} (2024).
\newblock \myhref{arXiv:2401.01797}{https://arxiv.org/abs/2401.01797}.
\newblock
  \myhref{https://arxiv.org/abs/2401.01797}{https://arxiv.org/abs/2401.01797}.

\bibitem[BCH{\etalchar{+}}25]{Tindel_Big}
\textsc{F.~Baudoin}, \textsc{L.~Chen}, \textsc{C.-H. Huang},
  \textsc{C.~Ouyang}, \textsc{S.~Tindel}, and \textsc{J.~Wang}.
\newblock {Weighted Besov spaces on Heisenberg groups and applications to the
  Parabolic Anderson model}, 2025.
\newblock \myhref{arXiv:2501.04593}{https://arxiv.org/abs/2501.04593}.
\newblock
  \myhref{https://arxiv.org/abs/2501.04593}{https://arxiv.org/abs/2501.04593}.

\bibitem[BFKG12]{Gallager_phase}
\textsc{H.~Bahouri}, \textsc{C.~Fermanian-Kammerer}, and \textsc{I.~Gallagher}.
\newblock Phase-space analysis and pseudodifferential calculus on the
  {H}eisenberg group.
\newblock \emph{Ast\'erisque} , no. 342, (2012), vi+127.

\bibitem[BG01]{Tindel_ref_gallager}
\textsc{H.~Bahouri} and \textsc{I.~Gallagher}.
\newblock Paraproduit sur le groupe de {H}eisenberg et applications.
\newblock \emph{Rev. Mat. Iberoamericana} \textbf{17}, no.~1, (2001), 69--105.
\newblock \myhref{doi:10.4171/RMI/289}{https://doi.org/10.4171/RMI/289}.

\bibitem[BGHZ22]{BGHZ22}
\textsc{Y.~Bruned}, \textsc{F.~Gabriel}, \textsc{M.~Hairer}, and
  \textsc{L.~Zambotti}.
\newblock Geometric stochastic heat equations.
\newblock \emph{Journal of the American Mathematical Society} \textbf{35},
  no.~1, (2022), 1--80.

\bibitem[BHZ19]{BHZ19}
\textsc{Y.~Bruned}, \textsc{M.~Hairer}, and \textsc{L.~Zambotti}.
\newblock Algebraic renormalisation of regularity structures.
\newblock \emph{Invent. Math.} \textbf{215}, no.~3, (2019), 1039--1156.
\newblock
  \myhref{doi:10.1007/s00222-018-0841-x}{https://doi.org/10.1007/s00222-018-0841-x}.

\bibitem[BOTW23]{Tindel_ito}
\textsc{F.~Baudoin}, \textsc{C.~Ouyang}, \textsc{S.~Tindel}, and
  \textsc{J.~Wang}.
\newblock Parabolic {A}nderson model on {H}eisenberg groups: the {I}t\^o{}
  setting.
\newblock \emph{J. Funct. Anal.} \textbf{285}, no.~1, (2023), Paper No. 109920,
  44.
\newblock
  \myhref{doi:10.1016/j.jfa.2023.109920}{https://doi.org/10.1016/j.jfa.2023.109920}.

\bibitem[BSS25]{BSS25}
\textsc{L.~Broux}, \textsc{H.~Singh}, and \textsc{R.~Steele}.
\newblock {Renormalised models for variable coefficient singular SPDEs}.
\newblock \emph{arXiv preprint} (2025).
\newblock \myhref{arXiv:2507.06851}{https://arxiv.org/abs/2507.06851}.

\bibitem[CCHS22]{CCHS22}
\textsc{A.~Chandra}, \textsc{I.~Chevyrev}, \textsc{M.~Hairer}, and
  \textsc{H.~Shen}.
\newblock Langevin dynamic for the 2{D} {Y}ang-{M}ills measure.
\newblock \emph{Publ. Math. Inst. Hautes \'{E}tudes Sci.} \textbf{136}, (2022),
  1--147.
\newblock
  \myhref{doi:10.1007/s10240-022-00132-0}{https://doi.org/10.1007/s10240-022-00132-0}.

\bibitem[CCHS24]{CCHS22b}
\textsc{A.~Chandra}, \textsc{I.~Chevyrev}, \textsc{M.~Hairer}, and
  \textsc{H.~Shen}.
\newblock Stochastic quantisation of {Y}ang-{M}ills-{H}iggs in 3{D}.
\newblock \emph{Invent. Math.} \textbf{237}, no.~2, (2024), 541--696.
\newblock
  \myhref{doi:10.1007/s00222-024-01264-2}{https://doi.org/10.1007/s00222-024-01264-2}.

\bibitem[CFX26]{CFX26}
\textsc{Y.~Chen}, \textsc{B.~Fehrman}, and \textsc{W.~Xu}.
\newblock Periodic homogenisation for two dimensional generalised parabolic
  {A}nderson model.
\newblock \emph{Probab. Theory Related Fields} \textbf{194}, no. 3-4, (2026),
  1235--1286.
\newblock
  \myhref{doi:10.1007/s00440-026-01465-1}{https://doi.org/10.1007/s00440-026-01465-1}.

\bibitem[CH16]{CH16}
\textsc{A.~Chandra} and \textsc{M.~Hairer}.
\newblock An analytic {BPHZ} theorem for regularity structures.
\newblock \emph{arXiv preprint} (2016).
\newblock \myhref{arXiv:1612.08138}{https://arxiv.org/abs/1612.08138}.

\bibitem[Clo25]{clozeau}
\textsc{N.~Clozeau}.
\newblock An inductive approach to stochastic estimates for the
  $\varphi^{4}_2$-equation with correlated coefficient field.
\newblock \emph{arXiv preprint} (2025).
\newblock \myhref{arXiv:2509.11309}{https://arxiv.org/abs/2509.11309}.
\newblock
  \myhref{https://arxiv.org/abs/2509.11309}{https://arxiv.org/abs/2509.11309}.

\bibitem[CR24]{LPtheorem}
\textsc{D.~Cardona} and \textsc{M.~Ruzhansky}.
\newblock Littlewood-{P}aley theorem, {N}ikolskii inequality, {B}esov spaces,
  {F}ourier and spectral multipliers on graded {L}ie groups.
\newblock \emph{Potential Anal.} \textbf{60}, no.~3, (2024), 965--1005.
\newblock
  \myhref{doi:10.1007/s11118-023-10076-7}{https://doi.org/10.1007/s11118-023-10076-7}.

\bibitem[CS26a]{CSYYM}
\textsc{I.~Chevyrev} and \textsc{H.~Shen}.
\newblock Invariant measure and universality of the 2{D} {Y}ang-{M}ills
  {L}angevin dynamic.
\newblock \emph{Comm. Pure Appl. Math.} \textbf{79}, no.~8, (2026), 1973--2102.
\newblock \myhref{doi:10.1002/cpa.70043}{https://doi.org/10.1002/cpa.70043}.

\bibitem[CS26b]{CS25}
\textsc{N.~Clozeau} and \textsc{H.~Singh}.
\newblock Renormalisation of singular {SPDE}s with correlated coefficients.
\newblock \emph{SIGMA Symmetry Integrability Geom. Methods Appl.} \textbf{22},
  (2026), Paper No. 021, 35.
\newblock
  \myhref{doi:10.3842/SIGMA.2026.021}{https://doi.org/10.3842/SIGMA.2026.021}.

\bibitem[CX23]{CX23}
\textsc{Y.~Chen} and \textsc{W.~Xu}.
\newblock Periodic homogenisation for $p(\phi)_2$ (2023).
\newblock \myhref{arXiv:2311.16398}{https://arxiv.org/abs/2311.16398}.
\newblock
  \myhref{https://arxiv.org/abs/2311.16398}{https://arxiv.org/abs/2311.16398}.

\bibitem[CYY26]{fractals_2}
\textsc{H.~Chen}, \textsc{Yifan}, and \textsc{Yang}.
\newblock Wick renormalized parabolic stochastic quantization equations on
  rough metric measure spaces (2026).
\newblock \myhref{arXiv:2605.05442}{https://arxiv.org/abs/2605.05442}.
\newblock
  \myhref{https://arxiv.org/abs/2605.05442}{https://arxiv.org/abs/2605.05442}.

\bibitem[DDD19]{DDD19}
\textsc{A.~Dahlqvist}, \textsc{J.~Diehl}, and \textsc{B.~K. Driver}.
\newblock The parabolic {A}nderson model on {R}iemann surfaces.
\newblock \emph{Probab. Theory Related Fields} \textbf{174}, no. 1-2, (2019),
  369--444.
\newblock
  \myhref{doi:10.1007/s00440-018-0857-6}{https://doi.org/10.1007/s00440-018-0857-6}.

\bibitem[DHZ94]{Nice_noteDHZ94}
\textsc{J.~Dziuba\'nski}, \textsc{W.~Hebisch}, and \textsc{J.~Zienkiewicz}.
\newblock Note on semigroups generated by positive {R}ockland operators on
  graded homogeneous groups.
\newblock \emph{Studia Math.} \textbf{110}, no.~2, (1994), 115--126.
\newblock
  \myhref{doi:10.4064/sm-110-2-115-126}{https://doi.org/10.4064/sm-110-2-115-126}.

\bibitem[FH20]{friz_hairer_20_introduction}
\textsc{P.~Friz} and \textsc{M.~Hairer}.
\newblock \emph{A Course on Rough Paths: With an Introduction to Regularity
  Structures}.
\newblock Universitext. Springer International Publishing, 2020.

\bibitem[FM12]{Führ}
\textsc{H.~F\"uhr} and \textsc{A.~Mayeli}.
\newblock Homogeneous {B}esov spaces on stratified {L}ie groups and their
  wavelet characterization.
\newblock \emph{J. Funct. Spaces Appl.} (2012), Art. ID 523586, 41.
\newblock
  \myhref{doi:10.1155/2012/523586}{https://doi.org/10.1155/2012/523586}.

\bibitem[FMV06]{FurioliMelziVeneruso}
\textsc{G.~Furioli}, \textsc{C.~Melzi}, and \textsc{A.~Veneruso}.
\newblock Littlewood-{P}aley decompositions and {B}esov spaces on {L}ie groups
  of polynomial growth.
\newblock \emph{Math. Nachr.} \textbf{279}, no. 9-10, (2006), 1028--1040.
\newblock
  \myhref{doi:10.1002/mana.200510409}{https://doi.org/10.1002/mana.200510409}.

\bibitem[Fol75]{folland_75_subelliptic}
\textsc{G.~B. Folland}.
\newblock Subelliptic estimates and function spaces on nilpotent {L}ie groups.
\newblock \emph{Ark. Mat.} \textbf{13}, no.~2, (1975), 161--207.
\newblock \myhref{doi:10.1007/BF02386204}{https://doi.org/10.1007/BF02386204}.

\bibitem[FR16]{fischer_ruzhansky_18_quantisation}
\textsc{V.~Fischer} and \textsc{M.~Ruzhansky}.
\newblock \emph{Quantization on nilpotent {L}ie groups}, vol. 314 of
  \emph{Progress in Mathematics}.
\newblock Birkh\"{a}user/Springer, [Cham], 2016,  xiii+557.
\newblock
  \myhref{doi:10.1007/978-3-319-29558-9}{https://doi.org/10.1007/978-3-319-29558-9}.

\bibitem[FS82]{folland_stein_82_hardy}
\textsc{G.~B. Folland} and \textsc{E.~M. Stein}.
\newblock \emph{Hardy spaces on homogeneous groups}, vol.~28 of
  \emph{Mathematical Notes}.
\newblock Princeton University Press, Princeton, N.J.; University of Tokyo
  Press, Tokyo, 1982,  xii+285.

\bibitem[GH19]{GerasimovicsHairer}
\textsc{A.~Gerasimovi\v{c}s} and \textsc{M.~Hairer}.
\newblock H\"ormander's theorem for semilinear {SPDE}s.
\newblock \emph{Electron. J. Probab.} \textbf{24}, (2019), Paper No. 132, 56.
\newblock \myhref{doi:10.1214/19-ejp387}{https://doi.org/10.1214/19-ejp387}.

\bibitem[GIP15]{gubinelli_imkeller_perkowski_15}
\textsc{M.~Gubinelli}, \textsc{P.~Imkeller}, and \textsc{N.~Perkowski}.
\newblock Paracontrolled distributions and singular {PDE}s.
\newblock \emph{Forum Math. Pi} \textbf{3}, (2015), 75.
\newblock \myhref{doi:10.1017/fmp.2015.2}{https://doi.org/10.1017/fmp.2015.2}.

\bibitem[GS12]{gallager_sire}
\textsc{I.~Gallagher} and \textsc{Y.~Sire}.
\newblock Besov algebras on {L}ie groups of polynomial growth.
\newblock \emph{Studia Math.} \textbf{212}, no.~2, (2012), 119--139.
\newblock \myhref{doi:10.4064/sm212-2-2}{https://doi.org/10.4064/sm212-2-2}.

\bibitem[GT10]{Gub_tindel_sewing}
\textsc{M.~Gubinelli} and \textsc{S.~Tindel}.
\newblock Rough evolution equations.
\newblock \emph{Ann. Probab.} \textbf{38}, no.~1, (2010), 1--75.
\newblock \myhref{doi:10.1214/08-AOP437}{https://doi.org/10.1214/08-AOP437}.

\bibitem[Hai14]{Hai14}
\textsc{M.~Hairer}.
\newblock A theory of regularity structures.
\newblock \emph{Inventiones mathematicae} \textbf{198}, no.~2, (2014),
  269--504.
\newblock
  \myhref{doi:10.1007/s00222-014-0505-4}{https://doi.org/10.1007/s00222-014-0505-4}.

\bibitem[HL15]{hairer_labbe_15_simple}
\textsc{M.~Hairer} and \textsc{C.~Labb\'{e}}.
\newblock A simple construction of the continuum parabolic {A}nderson model on
  {$\mbR^2$}.
\newblock \emph{Electron. Commun. Probab.} \textbf{20}, (2015), no. 43, 11.
\newblock
  \myhref{doi:10.1214/ECP.v20-4038}{https://doi.org/10.1214/ECP.v20-4038}.

\bibitem[HRRvV25]{hu2025besovtriebellizorkinspaceshomogeneous}
\textsc{G.~Hu}, \textsc{D.~Rottensteiner}, \textsc{M.~Ruzhansky}, and
  \textsc{J.~T. van Velthoven}.
\newblock {Besov and Triebel-Lizorkin spaces on homogeneous groups}, 2025.
\newblock \myhref{arXiv:2501.08997}{https://arxiv.org/abs/2501.08997}.
\newblock
  \myhref{https://arxiv.org/abs/2501.08997}{https://arxiv.org/abs/2501.08997}.

\bibitem[HS90]{HebischSikora90}
\textsc{W.~Hebisch} and \textsc{A.~Sikora}.
\newblock A smooth subadditive homogeneous norm on a homogeneous group.
\newblock \emph{Studia Math.} \textbf{96}, no.~3, (1990), 231--236.
\newblock
  \myhref{doi:10.4064/sm-96-3-231-236}{https://doi.org/10.4064/sm-96-3-231-236}.

\bibitem[HS23]{HS23manifolds}
\textsc{M.~Hairer} and \textsc{H.~Singh}.
\newblock Regularity structures on manifolds and vector bundles.
\newblock \emph{arXiv preprint} (2023).
\newblock \myhref{arXiv:2308.05049}{https://arxiv.org/abs/2308.05049}.

\bibitem[HS25a]{HSper2}
\textsc{M.~Hairer} and \textsc{H.~Singh}.
\newblock Homogenisation of singular {SPDE}s.
\newblock \emph{arXiv preprint} (2025).
\newblock \myhref{arXiv:2510.19339}{https://arxiv.org/abs/2510.19339}.

\bibitem[HS25b]{HSper}
\textsc{M.~Hairer} and \textsc{H.~Singh}.
\newblock Periodic space-time homogenisation of the $\phi^4_2$ equation.
\newblock \emph{Journal of Functional Analysis} \textbf{288}, no.~5, (2025),
  110762.
\newblock
  \myhref{doi:10.1016/j.jfa.2024.110762}{https://doi.org/10.1016/j.jfa.2024.110762}.

\bibitem[HZZZ24]{KineticZhu}
\textsc{Z.~Hao}, \textsc{X.~Zhang}, \textsc{R.~Zhu}, and \textsc{X.~Zhu}.
\newblock Singular kinetic equations and applications.
\newblock \emph{Ann. Probab.} \textbf{52}, no.~2, (2024), 576--657.
\newblock \myhref{doi:10.1214/23-aop1666}{https://doi.org/10.1214/23-aop1666}.

\bibitem[Mey92]{meyer1992wavelets}
\textsc{Y.~Meyer}.
\newblock \emph{Wavelets and operators}.
\newblock No.~37. Cambridge university press, 1992.

\bibitem[MS25]{MS25}
\textsc{A.~Mayorcas} and \textsc{H.~Singh}.
\newblock Singular {SPDE}s on homogeneous lie groups.
\newblock \emph{Proceedings of the Royal Society of Edinburgh: Section A
  Mathematics} (2025), 1–72.
\newblock \myhref{doi:10.1017/prm.2025.1}{https://doi.org/10.1017/prm.2025.1}.

\bibitem[Sak79]{Besov_def}
\textsc{K.~Saka}.
\newblock {Besov spaces and Sobolev spaces on a nilpotent Lie group}.
\newblock \emph{Tohoku Mathematical Journal} \textbf{31}, no.~4, (1979), 383 --
  437.
\newblock
  \myhref{doi:10.2748/tmj/1178229728}{https://doi.org/10.2748/tmj/1178229728}.

\bibitem[Sin25]{Sin23}
\textsc{H.~Singh}.
\newblock {Canonical solutions to non-translation invariant singular SPDEs}.
\newblock \emph{Electronic Journal of Probability} \textbf{30}, (2025), 1--24.
\newblock \myhref{doi:10.1214/25-EJP1315}{https://doi.org/10.1214/25-EJP1315}.

\bibitem[tR98]{Elst_Robinson}
\textsc{A.~{ter Elst}} and \textsc{D.~W. Robinson}.
\newblock {Weighted Subcoercive Operators on Lie Groups}.
\newblock \emph{Journal of Functional Analysis} \textbf{157}, no.~1, (1998),
  88--163.
\newblock
  \myhref{doi:10.1006/jfan.1998.3259}{https://doi.org/10.1006/jfan.1998.3259}.

\end{thebibliography}

\end{document}